\documentclass[12pt,twoside]{amsart}
\usepackage[margin=3cm]{geometry}
\usepackage[colorlinks=false]{hyperref}
\usepackage[english]{babel}
\usepackage{graphicx,titling}
\usepackage{float}
\usepackage{amsmath,amsfonts,amssymb,amsthm}
\usepackage{lipsum}
\usepackage[T1]{fontenc}
\usepackage{fourier}
\usepackage{color}
\usepackage[latin1]{inputenc}
\usepackage{esint}
\usepackage{caption}
\usepackage{ccicons}

\makeatletter
\def\blfootnote{\xdef\@thefnmark{}\@footnotetext}
\makeatother

\newcommand\ccnote{
    \blfootnote{\ccLogo\, \ccAttribution\,\, Licensed under a Creative Commons Attribution License (CC-BY).}
}

\usepackage[export]{adjustbox}
\numberwithin{equation}{section}
\usepackage{setspace}

\renewcommand{\leq}{\leqslant}

\renewcommand{\geq}{\geqslant}
\renewcommand{\mathbb}{\varmathbb}
\usepackage{fancyhdr}
\newtheorem{theorem}{Theorem}[section]
\newtheorem{lemma}[theorem]{Lemma}
\newtheorem{corollary}[theorem]{Corollary}
\newtheorem{proposition}[theorem]{Proposition}
\newtheorem{definition}[theorem]{Definition}
\newtheorem{remark}[theorem]{Remark}
\usepackage{mathtools}
\newtheorem{example}[theorem]{Example}
\newtheorem{examples}[theorem]{Examples}

\newcommand\R{\mathbb{R}}
\newcommand\C{\mathbb{C}}

\newcommand\Z{\mathbb{Z}}

\newcommand\Energy{{\mathcal{E}}}
\newcommand\eps{\varepsilon}

\address{Laura Cladek, University of California, Los Angeles, Department of Mathematics, 405 Hilgard Ave, Los Angeles CA 90095}
\email{cladekl@math.ucla.edu}
\medskip
\address{Terence Tao, University of California, Los Angeles, Department of Mathematics, 405 Hilgard Ave, Los Angeles CA 90095} 
\email{tao@math.ucla.edu}

\begin{document}

\thispagestyle{empty}

\begin{minipage}{0.28\textwidth}
\begin{figure}[H]
\includegraphics[width=2.5cm,height=2.5cm,left]{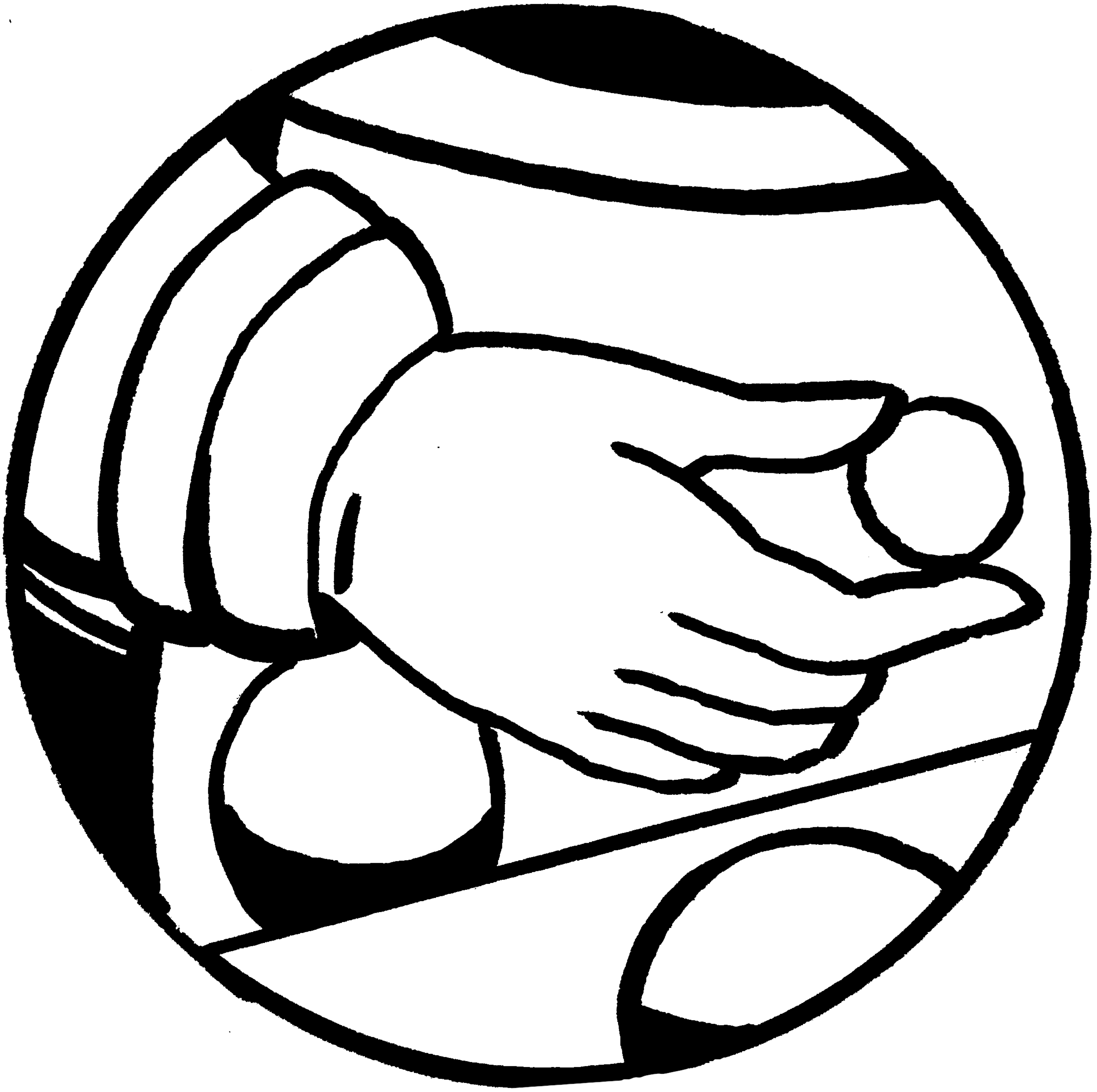}
\end{figure}
\end{minipage}
\begin{minipage}{0.7\textwidth} 
\begin{flushright}
Ars Inveniendi Analytica (2021), Paper No. 1, 38 pp.
\\
DOI 10.15781/gw9q-k252
\end{flushright}
\end{minipage}

\ccnote

\vspace{1cm}


\begin{center}
\begin{huge}
\textit{Additive energy of regular measures in one and higher dimensions, and the fractal uncertainty principle}

\end{huge}
\end{center}

\vspace{1cm}


\begin{minipage}[t]{.28\textwidth}
\begin{center}
{\large{\bf{Laura Cladek}}} \\
\vskip0.15cm
\footnotesize{University of California, Los Angeles}
\end{center}
\end{minipage}
\hfill
\noindent
\begin{minipage}[t]{.28\textwidth}
\begin{center}
{\large{\bf{Terence Tao}}} \\
\vskip0.15cm
\footnotesize{University of California, Los Angeles}
\end{center}
\end{minipage}

\vspace{1cm}

\begin{center}
\noindent \em{Communicated by Larry Guth}
\end{center}
\vspace{1cm}


\noindent \textbf{Abstract.} \textit{We obtain new bounds on the additive energy of (Ahlfors--David type) regular measures in both one and higher dimensions, which implies expansion results for sums and products of the associated regular sets, as well as more general nonlinear functions of these sets.  As a corollary of the higher-dimensional results we obtain some new cases of the fractal uncertainty principle in odd dimensions. 
}
\vskip0.3cm

\noindent \textbf{Keywords.} Fractal uncertainty principle, additive energy, Ahlfors--David regular sets.

\vspace{0.5cm}


\section{Introduction}

This paper is primarily concerned with the estimation of additive energies of regular measures (in the sense of Ahlfors and David, localized to specific ranges of scales) on Euclidean spaces $\R^d$.  We recall the key definitions.  If $x \in \R^d$ and $r>0$, we let $B(x,r) = B^d(x,r) \coloneqq \{ y \in \R^d: |x-y| \leq r \}$ denote the closed ball of radius $r$ centered at $x$, where we use $||$ for the Euclidean norm on $\R^d$.

\begin{definition}[Regular sets]\label{reg-set}\cite{bourgain-dyatlov}  Let $0 \leq \delta \leq d$ with the ambient dimension $d$ an integer, let $0 < \alpha_0 < \alpha_1$ be scales, and let $C \geq 1$.  A closed non-empty subset $X$ of $\R^d$ is said to be \emph{$\delta$-regular} on scales $[\alpha_0, \alpha_1]$ with constant $C$ if there exists a Radon measure\footnote{All measures in this paper will be unsigned.} $\mu_X$ (which we call a \emph{$\delta$-regular measure} or simply \emph{regular measure} associated to $X$) obeying the following properties:
\begin{itemize}
\item[(i)] $\mu_X$ is supported on $X$ (thus $\mu_X(\R^d \backslash X) = 0$);
\item[(ii)] For every ball $B(x,r)$ of radius $r \in [\alpha_0,\alpha_1]$, we have the upper bound $\mu_X(B(x,r)) \leq C r^\delta$;
\item[(iii)] If in addition in (i) we have $x \in X$, then we have the matching lower bound $\mu_X(B(x,r)) \geq C^{-1} r^\delta$.
\end{itemize}
\end{definition}

Examples of regular sets include Cantor sets, smooth compact submanifolds of $\R^d$, and $r$-neighbourhoods of such sets for $0 < r \leq \alpha_0$; see for instance \cite{dyatlov-survey} for such examples and further discussion.  Now we review the notion of additive energy at a given scale, as given for instance in \cite{dyatlov-zahl}.

\begin{definition}[Additive energy]\label{energy-def} Let $\mu$ be a Radon measure on $\R^d$, and let $r>0$.  The \emph{additive energy} $\Energy(\mu,r)$ of $\mu$ at scale $r$ is defined to be the quantity
\begin{equation}\label{energy-form}
\Energy(\mu,r) \coloneqq \mu^4\left( \{ (x_1,x_2,x_3,x_4) \in (\R^d)^4: |x_1+x_2-x_3-x_4| \leq r \}\right)
\end{equation}
where $\mu^4$ is the product measure on $(\R^d)^4$ formed from four copies of the measure $\mu$ on $\R^d$.
\end{definition}

By the Fubini--Tonelli theorem one can write
$$\Energy(\mu,r) = \int_{\R^d} \int_{\R^d} \int_{\R^d} \mu( B(x_1+x_2-x_3,r))\ d\mu(x_1) d\mu(x_2) d\mu(x_3)$$
so we have the trivial bound
\begin{equation}\label{triv-bound}
\Energy(\mu,r) \leq \mu(\R^d)^3 \sup_{x \in \R^d} \mu(B(x,r)).
\end{equation}
In particular, if $\mu = \mu_X$ is a regular measure associated to a $\delta$-regular set on scales $[\alpha,1]$ with constant $C$ that is supported on the unit ball $B(0,1)$, then we have the ``trivial bound''
\begin{equation}\label{triv}
\Energy(\mu,r) \leq C^4 r^\delta
\end{equation}
for any $\alpha \leq r \leq 1$.  

In the case when $\delta$ is an integer, this trivial bound can be sharp up to multiplicative constants.  Indeed, if $\mu = \mu_X$ is $\delta$-dimensional Lebesgue measure on the disk $X \coloneqq B^\delta(0,1) \times \{0\}^{d-\delta}$, then for any $0 < \alpha \leq r \leq 1$, one can verify that $\mu$ is regular on scales $[\alpha,1]$ with some constant\footnote{See Section \ref{notation-sec} for our asymptotic notation conventions.} $C = O_{d,\delta}(1)$, but that $\Energy(\mu,r) \sim_{d,\delta} r^\delta$.

However, when the dimension $\delta$ is not an integer one expects an improvement to the trivial bound \eqref{triv} in the asymptotic regime when $r$ goes to zero.  In one dimension $d=1$ this was achieved by Dyatlov and Zahl \cite{dyatlov-zahl}:

\begin{theorem}[Improved additive energy bound for regular measures in $\R$]\label{dz-main}\cite[Proposition 6.23]{dyatlov-zahl}  Let $0 < \delta < 1$, $C>1$ and $0 < \alpha < 1$.  Let $X \subset [0,1]$ be a $\delta$-regular set on scales $[\alpha,1]$, and let $\mu_X$ be an associated regular measure.  Then we have
$$ \Energy(\mu, r) \lesssim_{C,\delta} r^{\delta + \beta}$$
for any $\alpha \leq r < 1$, where $\beta > 0$ is of the form
\begin{equation}\label{beta}
\beta =\delta \exp\left( - K (1-\delta)^{-14} (1 + \log^{14} C) \right)
\end{equation}
for some absolute constant $K$.
\end{theorem}

Among other things, this theorem was used to establish the first non-trivial case of the \emph{fractal uncertainty principle}, which we will return to later in this introduction.  The proof of Theorem \ref{dz-main} relied on an induction on scale argument and some major inverse theorems in additive combinatorics, such as the Bogulybov--Ruzsa theorem of Sanders \cite{sanders-br}.  Ultimately, the key point is that the $\delta$-regular set $X$ is ``porous'' and cannot exhibit approximate translation invariance along a medium-length arithmetic progression; on the other hand, additive combinatorics and induction on scales can be used to produce such an approximate translation invariance if $\mu$ has exceptionally high additive energy.

The bound \eqref{beta} on $\beta$ behaves quasipolynomially in $C$.  In \cite[\S 6.8.3]{dyatlov-zahl} the question was raised as to whether the exponent $\beta$ could be improved to be polynomial in $C$.  Our first main result answers this question in the affirmative.

\begin{theorem}[Further improvement to additive energy bounds for regular measures in $\R$]\label{main}  With the hypotheses of Theorem \ref{dz-main}, one can take $\beta$ to be of the form
$$ \beta = c \min(\delta,1-\delta) C^{-25}$$
for some absolute constant $c>0$.
\end{theorem}

The exponent $25$ can be improved here, but we do not attempt to optimize it in this paper.  It may be possible to improve the bound\footnote{Since the initial release of this manuscript, Brendan Murphy (private communication) has shown us an unpublished manuscript using the induction on scales method of Bourgain \cite{bourgain} and Freiman's theorem to obtain a similar result with $\beta$ comparable to $\frac{1}{\log C \log\log C}$.} even further; the best known counterexample (see \cite[\S 6.8.3]{dyatlov-zahl}) only shows that $\beta$ must decay at least as fast as $\frac{1}{\log C}$ as $C \to \infty$ (holding $\delta$ fixed).

Our argument is elementary and relies heavily on the order structure of the real line $\R$, as well as the ability to work with a mesh of dyadic (or more precisely, $K$-adic for a large $K$) intervals.  Roughly speaking, the idea is to identify a lot of ``left\footnote{One could also work just as easily with a notion of ``right edge'' if desired by switching all the signs.} edges'' of (a suitable discretization of) the $\delta$-regular set $X$: ($K$-adic) intervals $I = [x,x+K^{-2n})$ which intersect the set $X$, but such that there is a large interval $[x-K^{-2n+1},x)$ immediately to the left of $I$ that is completely disjoint from $X$.  Informally, left edges identify locations and scales where the set $X$ visibly fails to behave like an additive subgroup of the real numbers.  As $\delta$ is bounded away from $0$ and $1$, we will be able to exhibit many left edges at many scales, to the point that almost all of the elements of $X$ will be contained in at least one left edge (or slight technical generalization of this concept which we call a left near-edge).  On the other hand, if one considers a sum $x_1+x_2$ arising from a pair $x_1,x_2 \in X$ that are contained in left-edges $[y_1,y_1+K^{-2n})$, $[y_2,y_2+K^{-2n})$ respectively, then there are unusually few other pairs $x_3,x_4 \in X$ with $x_3, x_4$ close (within $K^{-2n+1}/10$ say) of $x_1,x_2$ respectively, such that $x_1+x_2 \approx x_3+x_4$. This is because in the vicinity of the left-edge $[y_1,y_1+K^{-2n})$, most of the candidates for $x_3$ lie far to the right of $x_1$, and similarly most of the candidates of $x_4$ lie far to the right of $x_2$, so it is difficult to keep $x_3+x_4$ close to $x_1+x_2$; see Figure \ref{fig:leftedge}.  Because of this, each pair of left-edges can be used to create a slight diminution of the additive energy; by combining the effect of all the available pairs of left-edges at all scales, we can obtain a preliminary improvement to the trivial bound \eqref{triv} on the additive energy at some fixed small scale $r_0$ (made precise in Proposition \ref{slight-gain-1} below), which can then be iterated by a standard ``induction on scales'' argument to produce Theorem \ref{main}.  Our bounds are superior to that in Theorem \ref{dz-main} because they do not rely on the Bogulybov--Ruzsa theorem of Sanders \cite{sanders-br}, for which the best known bounds are only quasipolynomial in nature.

\begin{figure} [t]
\centering
\includegraphics[width=4in]{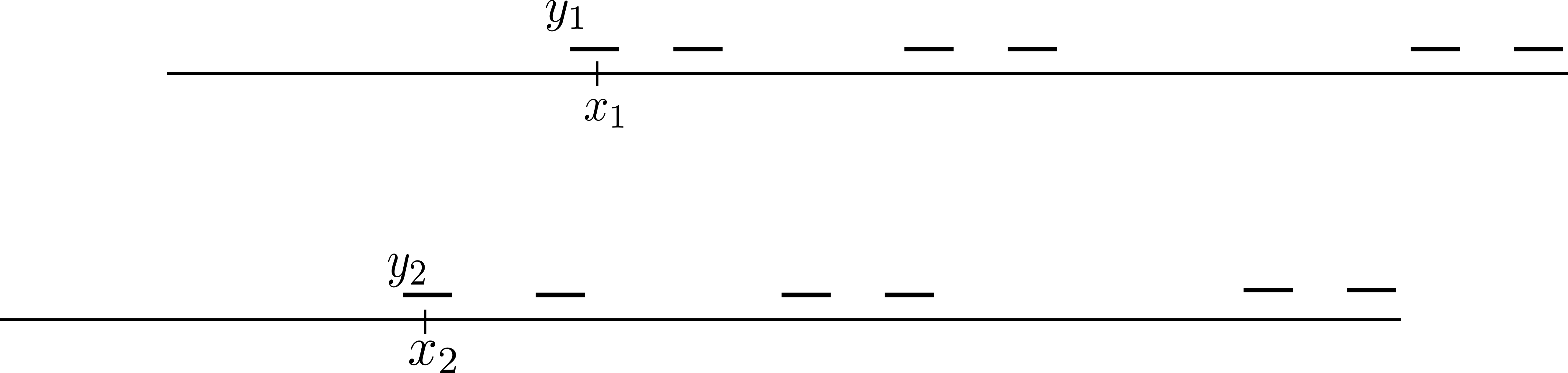}
\caption{The $\delta$-regular set $X$ can be covered by intervals of length $K^{-2n}$ for some given $n$, which are arranged in a fractal-type pattern.  Some of these intervals are depicted here, on two different portions of the real line $\R$.  The intervals $[y_1, y_1+K^{-2n})$ and $[y_2, y_2+K^{-2n})$ are left edges: they have nearby intervals covering $X$ to the right, but not to the left.  Note that if $x_1 \in X \cap [y_1 + y_1+K^{-2n})$ and $x_4 \in X \cap [y_4, y_4+K^{-2n})$ then there are a relatively small number of additional pairs $(x_3,x_4)$ of real numbers $x_3,x_4 \in X$ that are somewhat close to $x_1,x_2$ respectively such that $x_1+x_2 \approx x_3 + x_4$. This leads to a small but non-zero diminuition of the additive energy at the scales between to $K^{-2n}$ and $K^{-2n+1}$, which one can then hope to iterate.}
\label{fig:leftedge}
\end{figure}

Our second result addresses the higher dimensional case.  Here the order-theoretic arguments do not seem to be effective (except possibly in the very high-dimensional regime $d-1 < \delta < d$), and we revert to using more tools from additive combinatorics, and in particular the Bogulybov--Ruzsa theorem. As such, our bounds are of the same general quasipolynomial shape as in Theorem \ref{dz-main}, though it seems likely that further progress on the ``polynomial Freiman--Ruzsa conjecture'' (see e.g., \cite{sanders-bams}) may eventually lead to polynomial bounds along the lines of Theorem \ref{main}.

\begin{theorem}[Improved additive energy bound for higher dimensional regular measures]\label{main-second} Let $d \geq 1$ be an integer, let $0 < \delta < d$ be a non-integer, and let $C>1$ and $0 < \alpha < 1$.  Let $X \subset B^d(0,1)$ be a $\delta$-regular set on scales $[\alpha,1]$, and let $\mu = \mu_X$ be an associated regular measure.  Then we have
$$ \Energy(\mu, r) \lesssim_{C,\delta,d} r^{\delta + \beta}$$
for any $\alpha \leq r < 1$, where $\beta  = \beta_{d,\delta,C} > 0$ takes the quasipolynomial form
$$ \beta = \exp\left( - O_{\delta,d}( 1 + \log^{O_{\delta,d}(1)}(C) ) \right).$$
\end{theorem}

We prove this theorem in Section \ref{higher-sec}.  As with Theorem \ref{main}, it suffices to gain a small amount over the trivial bound at some fixed scale $r_0$ as one can then use induction on scales to then obtain the power saving for arbitrary scales $r$.  Our proofs shares some features in common with that in \cite{dyatlov-zahl}, in that one assumes that the energy is unusually large and uses this to locate a non-trivial arithmetic progression along which $\mu$ obeys an approximate symmetry in a certain range of scales.  If the dimension $\delta$ was less than $1$ then one could use the regularity to obtain a contradiction (basically because a $\delta$-dimensional measure cannot support an entire line segment if $\delta<1$).  The main novelty in our argument is the treatment of the higher dimensional case $\delta>1$.  Here the strategy is to ``quotient out'' by the arithmetic progression and obtain (at least on some range of scales) what is (morally at least) a $\delta-1$-regular measure supported on a $d-1$-dimensional hyperplane of $\R^d$.  This can then be treated by a suitable induction hypothesis (note how the crucial hypothesis that $\delta$ is not an integer is propagated by this process).  In order to make the notion of ``quotienting out'' by an arithmetic progression (which is merely an approximate subgroup of $\R^d$, as opposed to a genuine subgroup) rigorous, we will rely heavily on the machinery of additive combinatorics, most notably the theory of the Gowers uniformity norm $U^2$ and of approximate groups, and how one can split up the problem of estimating a global $U^2$ Gowers norm into the ``fine scale'' problem of estimating various local Gowers norms restricted to a ``coset'' $4H+y$ of some approximate group $H$, and the ``coarse scale'' problem of controlling the $U^2$ norm of the output of these local Gowers norms (viewed as a function of $y$); see Lemma \ref{relate}(iii) for a precise statement.

\begin{remark} It seems likely that the techniques in this paper can be combined with those in \cite{rossi-shmerkin} to obtain new $L^q$ improving bounds for convolutions of regular measures, but we will not pursue this direction here.
\end{remark}

\begin{remark}  Since the initial release of this manuscript, Pablo Shmerkin (private communication) has informed us that Hochman's inverse entropy theorem \cite[Theorem 2.10]{hochman} implies a qualitative entropy version of Theorem \ref{main-second}, which roughly speaking asserts that under the hypotheses of that theorem (holding at all scales between $0$ and $1$), $\mu * \mu$ should have strictly larger entropy than $\mu$ itself.  Very roughly speaking, this is because Hochman's theorem asserts that if this entropy increase does not hold, then at many scales $\mu$ must locally resemble an approximately uniform measure on affine subspaces of some integer dimension between $0$ and $d$, which is inconsistent with regularity if $\delta$ is not an integer.  In fact, it appears likely that these methods could give an alternate proof of Theorem \ref{main-second} itself, but without a quantitative value of $\beta$.
\end{remark}

One can use these additive energy bounds to obtain expansion estimates for both linear and nonlinear maps.  Here is a sample such result.  For any measurable $E \subset \R^d$ we let $|E|$ denote its Lebesgue measure.

\begin{theorem}[Nonlinear expansion]\label{nonlinear}  Let $d \geq 1$ be an integer, let $0 < \delta < d$ be a non-integer, and let $C>1$ and $0 < r < 1$.  Let $F: \R^d \times \R^d \to \R^d$ be a $C^2$ (twice continuously differentiable) map such that for any $x,y \in B(0,2)$, the derivative maps $D_x F(x,y), D_y F(x,y): \R^d \to \R^d$ are invertible.  Let $X,Y \subset B(0,1)$ be $\delta$-regular sets on scales $[r,1]$, and let $X_r \coloneqq X+B(0,r)$, $Y_r \coloneqq Y + B(0,r)$ be the $r$-neighborhoods of $X,Y$ respectively.  Then the set
$$ F(X_r, Y_r) \coloneqq \{ F(x,y): x \in X_r, y \in Y_r \}$$
has measure
$$ | F( X_r, Y_r )| \gtrsim_{C,\delta,d,F} r^{d-\delta-\beta}$$
where $\beta$ takes the form
$$ \beta = \exp\left( - O_{\delta,d,F}( 1 + \log^{O_{\delta,d}(1)}(C) ) \right).$$
When $d=1$, we can take
$$ \beta = c_F \min(\delta,1-\delta) C^{-25}$$
for some $c_F>0$ depending on $F$.
\end{theorem}

In particular, we have
\begin{equation}\label{x-sum}
|X_r + X_r|, |X_r - X_r | \gtrsim_{C,\delta,d} r^{d-\delta-\beta}
\end{equation}
and in one dimension we similarly have
\begin{equation}\label{x-product}
|X_r \cdot X_r| \gtrsim_{C,\delta} r^{1-\delta-\beta}
\end{equation}
if $X$ lies in (say) $[1/2,1]$ (this can be deduced either by taking logarithms, or by applying the above theorem to the multiplication map $x,y \mapsto xy$ after applying some mild changes of variable to avoid the degeneracies at $x,y=0$).  Here we use the usual sumset notations
\begin{align*}
A+B &\coloneqq \{ a+b: a \in A, b \in B \} \\
A-B &\coloneqq \{ a-b: a \in A, b \in B \} \\
A+b &\coloneqq \{ a+b: a \in A \} \\
-A &\coloneqq \{ -a: a \in A \} \\
A\cdot B &\coloneqq \{ ab: a \in A, b \in B \}.
\end{align*}
For comparison, it is easy to see from Definition \ref{reg-set} that
$$ |X_r| \sim_{C,\delta,d} r^{d-\delta}$$
so the bounds in \eqref{x-sum}, \eqref{x-product} represent a power gain over the trivial bound $|A+B| \geq |A|, |B|$ (and $|A \cdot B| \gtrsim |A|, |B|$ when $A,B \subset [1,2]$).

We establish Theorem \ref{nonlinear} in Section \ref{nonlinear-sec}.  This result can be compared with the celebrated discretized sum-product theorem of Bourgain \cite{bourgain}, which in our language would give a bound of the form
\begin{equation}\label{sum-product}
\max( |X_r+X_r|, |X_r \cdot X_r| ) \gtrsim_{C,\delta} r^{\delta-\beta}
\end{equation}
for some $\beta = \beta_\delta > 0$ and $X \subset [1,2]$, but where the lower bound in Definition \ref{reg-set}(iii) is only assumed to hold at scale $r=1$; see the recent preprint \cite{gkz} for an explicit value of $\beta$ as a function of $\delta$.  In this more general set up there are standard constructions (in which $X$ resembles either a long arithmetic progression or a long geometric progression at various scales) that show that one can no longer expect the separate bounds \eqref{x-sum}, \eqref{x-product}, and and can only hope for \eqref{sum-product}.  This is consistent with results such as \eqref{x-sum}, \eqref{x-product} because regular sets cannot support long arithmetic progressions or long geometric progressions.  Under the same general hypotheses on $X$ considered in the above-mentioned result \eqref{sum-product} of Bourgain, an expansion bound of the form
$$ |F( X_r, X_r )| \gtrsim_{C,\delta,F} r^{\delta-\beta}$$
was also recently established in \cite{rz} for polynomial maps $F$ under the (necessary) Elekes-R\'onyai condition that $F(x,y)$ is not of the form $h(a(x)+b(y))$ or $h(a(x)b(y))$ for some polynomials $h,a,b$.

\subsection{Application to the fractal uncertainty principle}

Let $d \geq 1$ be a dimension, and let $0 < h \leq 1$ be a small parameter.  We then define the semiclassical Fourier transform 
$${\mathcal F}_h f \colon L^2(\R^d) \to L^2(\R^d)$$ 
by the formula
$$ {\mathcal F}_h f(\xi) \coloneqq (2\pi h)^{-d/2} \int_{\R^d} e^{-i x \cdot \xi/h} f(x)\ dx$$
for Schwartz functions $f$, extended to $L^2(\R^d)$ by continuity in the usual fashion.  The \emph{fractal uncertainty principle} concerns operator norm estimates of the form
$$ \| 1_{X_h} {\mathcal F}_h 1_{Y_h} \|_{L^2(\R^d) \to L^2(\R^d)} \lesssim h^\sigma$$
for various sets $X_h,Y_h$ (which are permitted to depend on $h$) and various exponents $\sigma$.  Here and in the sequel we use $1_E$ to denote the indicator function of a set $E$.  A model case is when $X,Y$ are $\delta$-regular subsets of $B(0,1)$ at scales $[h,1]$ for some constant $C$, and $X_h \coloneqq X + B(0,h)$, $Y_h \coloneqq Y + B(0,h)$ are the $h$-neighborhoods of $X,Y$ respectively.  A standard application of Plancherel's theorem and trivial bound
$$ \| {\mathcal F}_h f\|_{L^\infty(\R^d)} \lesssim_d h^{-d/2} \|f\|_{L^1(\R^d)}$$
then gives the ``trivial bound''
\begin{equation}\label{trivial-unc}
 \| 1_{X_h} {\mathcal F}_h 1_{Y_h} \|_{L^2(\R^d) \to L^2(\R^d)} \lesssim_{C,d} h^{\max\left(\frac{d}{2}-\delta,0\right)}
 \end{equation}
 (see e.g., \cite{dyatlov-survey} for the argument in the one-dimensional case $d=1$, which extends without difficulty to higher dimensions).  In one dimension $d=1$ with $0 < \delta < 1$, the fractal uncertainty principle \cite{dyatlov-zahl}, \cite{dyatlov-jin}, \cite{bourgain-dyatlov} asserts that one can improve this bound to
 \begin{equation}\label{frac-unc}
 \| 1_{X_h} {\mathcal F}_h 1_{Y_h} \|_{L^2(\R) \to L^2(\R)} \lesssim_{C,\delta} h^{\max\left(\frac{1}{2}-\delta,0\right) + \beta}
 \end{equation}
for some $\beta>0$ depending only on $C,\delta$.  Indeed, the following specific values for $\beta$ are known:
\begin{itemize}
\item[(i)]  \cite[Theorems 4, 6]{dyatlov-zahl} One can take
\begin{equation}\label{beta-frac}
 \beta = \frac{3}{8} \left(\frac{1}{2}-\delta\right) - \max\left(\frac{1}{2}-\delta,0\right) + \frac{1}{16} \delta \exp( - K (1-\delta)^{14} (1+\log^{14} C) )
 \end{equation}
for some absolute constant $K>0$ (this only gives a positive value of $\beta$ for $\delta$ sufficiently close to $1/2$).
\item[(ii)] \cite[Theorem 1.2]{jin-zhang} For $\delta \geq 1/2$, one can take
 \begin{equation}\label{sss}
 \beta = \exp\left[ -\exp( K( C \delta^{-1} (1-\delta)^{-1})^{K(1-\delta)^{-2}} )\right]
 \end{equation}
for an absolute constant $K>0$.
\item[(iii)] \cite[Theorem 1, Remarks 1]{dyatlov-jin}  For $0 < \delta \leq 1/2$, one can take
$$ \beta = (5C)^{-\frac{160}{\delta(1-\delta)}}.$$
\end{itemize}
Fractal uncertainty principles can be applied to quantum chaos to obtain lower bounds on the mass of eigenfunctions and to produce spectral gaps on various negatively curved manifolds see \cite{dyatlov-survey} for a survey of this connection.  Fractal uncertainty principles have also been established for a wider class of sets than regular sets, and particular to \emph{porous sets}, but we will not discuss this generalization further here.  We remark that while the exponent \eqref{beta-frac} was obtained via the additive energy method, other fractal uncertainty principle estimates, such as the one giving \eqref{sss}, relied on somewhat different techniques, such as the Beurling--Malliavin multiplier theorem.

By combining \cite[Theorem 4]{dyatlov-zahl} (or \cite[Proposition 5.4]{dyatlov-survey}) with Theorem \ref{main}, we can immediately improve \eqref{beta-frac} to
$$
\beta = \frac{3}{8} \left(\frac{1}{2}-\delta\right) - \max\left(\frac{1}{2}-\delta,0\right) + c \min(\delta,1-\delta) C^{-25}
$$
for an absolute constant $c>0$; this is still only an improvement over the trivial bound for $\delta$ very close to $1/2$, but the dependence on the regularity constant $C$ has improved.

In higher dimension $d > 1$, a similar combination of \cite[Theorem 4]{dyatlov-zahl} and  Theorem \ref{main-second} gives the estimate 
$$
\| 1_{X_h} {\mathcal F}_h 1_{Y_h} \|_{L^2(\R^d) \to L^2(\R^d)} \lesssim_{C,d,\delta} h^{\max(\frac{d}{2}-\delta,0) + \beta}$$
with
\begin{equation}\label{bad}
\beta = \frac{3}{8} \left(\frac{d}{2}-\delta\right) - \max\left(\frac{d}{2}-\delta,0\right) + \exp\left( - O_{\delta,d}( 1 + \log^{O_{\delta,d}(1)}(C) ) \right)
\end{equation}
when $\delta$ is a non-integer.  In the case of even dimension this does not\footnote{In particular, our results do not directly make progress on \cite[Conjecture 6.2]{dyatlov-survey}, which concerns the two-dimensional situation $d=2$.} give a positive value of $\beta$ for any value of $\delta$, which is to be expected since the no improvement to the trivial exponent $\max(\frac{d}{2}-\delta,0)$ is possible at the critical value $\delta=d/2$ in this case (see \cite[Example 6.1]{dyatlov-survey} for a counterexample when $d=2$, and a similar construction also works in other even dimensions).  However when $d$ is odd, $d/2$ is a non-integer, and for $\delta$ sufficiently close to $d/2$ the implied constants in \eqref{bad} can then be verified to be uniform in $\delta$ (for $d$ fixed), leading to a fractal uncertainty principle of the form
$$ \beta = \frac{3}{8} \left(\frac{d}{2}-\delta\right) - \max\left(\frac{d}{2}-\delta,0\right) + \exp\left( - O_d( 1 + \log^{O_d(1)} C) \right)$$
which is non-negative for $\delta$ sufficiently close to $d/2$.  To the authors knowledge, this is the first higher dimensional fractal uncertainty principle that holds for \emph{arbitrary} regular sets.  In the case when one of the sets $X,Y$ was the Cartesian product of $d$ one-dimensional regular (or porous) sets, a higher dimensional fractal uncertainty principle (with exponents similar in shape to \eqref{sss}) was recently obtained  in \cite{han-schlag}.  Also, as observed in \cite[\S VI.A]{dyatlov-survey}, a higher dimensional fractal uncertainty principle can sometimes be deduced from iterating the one-dimensional principle if certain projections and fibers of $X,Y$ are assumed to be porous.

If one allows the sets $X,Y$ to be regular with different dimensional parameters $\delta, \delta'$ then the additive energy bounds in Theorems \ref{main}, \ref{main-second} can give further fractal uncertainty principles.  Indeed we have the following statement:

\begin{theorem}[Consequence of additive energy bounds]\label{add-eng}  
 Let $d \geq 1$ be an integer, let $0 < \delta < d$ be a non-integer, and let $C>1$ and $0 < h < 1$.  Let $Y \subset B^d(0,1)$ be a $\delta$-regular set on scales $[h,1]$ with constant $C$.  Then
\begin{equation}\label{fey}
\| {\mathcal F}_h 1_{Y_h} \|_{L^2(\R^d) \to L^8(\R^d)} \lesssim_{C,d,\delta} h^{-3\delta/8+\beta}
\end{equation}
where
$$ \beta = \exp\left( - O_{\delta,d}( 1 + \log^{O_{\delta,d}(1)}(C) ) \right);$$
in the $d=1$ case one can instead take
$$ \beta = c \min(\delta,1-\delta) C^{-25}$$
for some absolute constant $c>0$.  In particular, by H\"older's inequality, if $X \subset B^d(0,1)$ is $\delta'$-regular on scales $[h,1]$ with constant $C'$ for some $0 \leq \delta' \leq d$ and $C' > 1$, then
\begin{equation}\label{fey-2} \| 1_{X_h} {\mathcal F}_h 1_{Y_h} \|_{L^2(\R) \to L^2(\R)} \lesssim_{C,C',d,\delta,\delta'} h^{\frac{3}{8} (d-\delta-\delta')+\beta}
\end{equation}
\end{theorem}

We prove this theorem in Section \ref{fract-sec}.  One can compare \eqref{fey-2} with the trivial bound, which in this case is
\begin{equation}\label{xhyh}
\| 1_{X_h} {\mathcal F}_h 1_{Y_h} \|_{L^2(\R) \to L^2(\R)} \lesssim_{C,C',d,\delta,\delta'} h^{\max\left(\frac{d-\delta-\delta'}{2},0\right)}.
\end{equation}
Thus one has new fractal uncertainty principles for $\delta'$-regular $X$ and $\delta$-regular $Y$ when $\delta$ is non-integer and $\delta+\delta'$ is sufficiently close to $d$; by duality one also has similar results when it is $\delta'$ that is assumed to be non-integer in place of $\delta$. As before, a modification of \cite[Example 6.1]{dyatlov-survey} shows that no improvement of the trivial bound \eqref{xhyh} can be obtained when $\delta,\delta'$ are both integers.

\subsection{Acknowledgments}

The first author is supported by a National Science Foundation Postdoctoral Fellowship, NSF grant 1703715. The second author is supported by NSF grant DMS-1764034 and by a Simons Investigator Award. The first author thanks Semyon Dyatlov for helpful conversations, and we also thank Josh Zahl, the anonymous referees, and some anonymous comments on the second author's blog for corrections.

\subsection{Notation}\label{notation-sec}

We use the asymptotic notation $X \lesssim Y$, $Y \gtrsim X$, or $X = O(Y)$ to denote the bound $|X| \leq CY$ for an absolute constant $C > 0$, and use $X \sim Y$ synonymously with $X \lesssim Y \lesssim X$.  If we need the constant $C$ to depend on additional parameters (e.g., the dimension $d$), we indicate this by subscripts, thus for instance $X \lesssim_d Y$ denotes the bound $|X| \leq C_d Y$ for some constant $C_d>0$ that depends on $d$.

If $E$ is a finite set, we use $\# E$ to denote its cardinality.

\section{The one-dimensional case: obtaining a slight gain}

In this section we begin the proof of Theorem \ref{main}.  Namely, we establish the following seemingly weaker variant in which one obtains only a slight improvement over the trivial bound \eqref{triv}, but at a fixed scale $r_0>0$.  In the next section we will use standard induction on scales argument to iterate this slight improvement to a power gain in the scale parameter.

\begin{proposition}[Slight gain over the trivial bound in one dimension]\label{slight-gain-1}  Let $0 < \delta < 1$, $C>1$ and $0 < \eps \leq 1/2$.  Let $r_0>0$ be the quantity
\begin{equation}\label{r0-def} r_0 \coloneqq \exp\left( - C_2 \frac{C^{16} \log^2(C/\eps)}{\eps^2 \min(\delta,1-\delta)}\right)
\end{equation}
for a sufficiently large absolute constant $C_2$.
Let $X \subset \R$ be a $\delta$-regular set on scales $[r_0,1]$ with constant $C$, and let $\mu = \mu_X$ be an associated regular measure.  Then we have
$$ \Energy(\mu|_{[-1,1]}, r_0) \leq \eps r_0^\delta.$$
\end{proposition}

We now turn to the proof of this proposition.  Let $0 < \delta < 1$, $C>1$, $0 < \eps \leq 1/2$, and define $r_0$ by \eqref{r0-def}.  We let $X \subset \R^d$ be $\delta$-regular on scales $[r_0,1]$ with constant $C$ and associated regular measure $\mu_X$.  Our task is to establish the bound
$$ \Energy(\mu|_{[-1,1]}, r_0) \leq \eps r_0^\delta.$$

We first estimate the left-hand side by an integral involving the $\mu^2$-measure of various ``strips'' $S_z$:

\begin{lemma}\label{mor}  One has
\begin{equation}\label{emor}
\Energy(\mu|_{[-1,1]}, r_0) \lesssim r_0^{-1} \int_{[-3,3]} \mu^2( S_z )^2\ dz
\end{equation}
where $\mu^2 = \mu \times \mu$ is the product of two copies of the measure $\mu$, and for each $z \in \R$, $S_z$ denotes the strip
$$ S_z \coloneqq \{ (x,y) \in [-1,1]^2: |x+y-z| \leq r_0 \}$$
(see Figure \ref{fig:strip}).
\end{lemma}

\begin{figure} [t]
\centering
\includegraphics[width=4in]{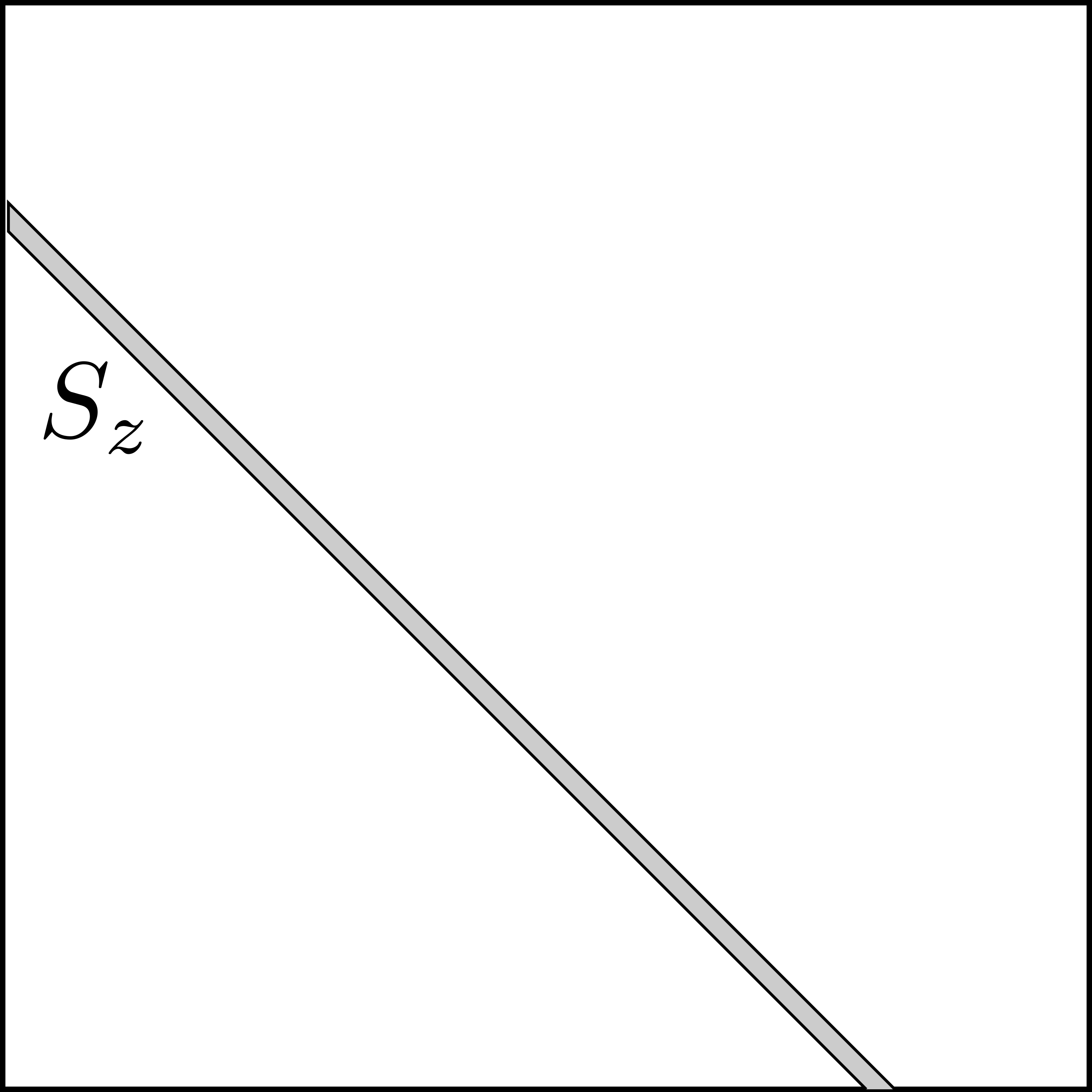}
\caption{A strip $S_z$.}
\label{fig:strip}
\end{figure}

\begin{proof}  By the Fubini--Tonelli theorem, the right-hand side is
$$ r_0^{-1} \int_{\R^4} \left(\int_{[-3,3]} 1_{|x_1+x_4-z|, |x_2+x_3-z| \leq r_0}\ dz\right)\ d\mu^4(x_1,x_2,x_3,x_4).$$
If $x_1,x_2,x_3,x_4 \in [-1,1]$ are such that $|x_1-x_2-x_3+x_4| \leq r_0$, then direct calculation shows that
$$ \int_{[-3,3]} 1_{|x_1+x_4-z|, |x_2+x_3-z| \leq r_0}\ dz \gtrsim r_0$$
and the claim follows.
\end{proof}

To estimate the right-hand side of \eqref{emor} we shall remove a small exceptional set $E$ from $[-1,1]^2$:

\begin{corollary}\label{conc}  For any measurable subset $E$ of $[-1,1]^2$, one has
$$ \Energy(\mu|_{[-1,1]}, r_0) \lesssim C^3 r_0^\delta \left( \mu^2(E) + \sup_{z \in [-3,3]} \mu(\pi(X^2 \cap S_z \backslash E)) \right)$$
where $\pi \colon \R^2 \to \R$ is the projection to the first coordinate: $\pi(x,y) \coloneqq x$.
\end{corollary}

\begin{proof}  Since $\mu^2$ is supported on $X^2$, we can split 
$$ \mu^2( S_z )^2 \leq \mu^2(S_z \cap E) \mu^2(S_z) +  \mu^2(X^2 \cap S_z \backslash E) \mu^2(S_z) $$
and thus
\begin{align*}
\Energy(\mu|_{[-1,1]}, r_0) &\lesssim r_0^{-1} \int_{[-3,3]} \mu^2( S_z \cap E ) \mu^2(S_z)\ dz\\
&\quad + r_0^{-1} \left(\sup_{z \in [-3,3]} \mu^2(X^2 \cap S_z \backslash E)\right) \int_{[-3,3]} \mu^2(S_z)\ dz.
\end{align*}
By the Fubini--Tonelli theorem one has
\begin{align*}
\int_{[-3,3]} \mu^2(S_z)\ dz &= \int_{[-1,1]^2} \int_{[-3,3]} 1_{|x+y-z| \leq r_0}\ dz d\mu^2(x,y) \\
&\lesssim r_0 \mu([-1,1])^2 \\
&\lesssim C^2 r_0
\end{align*}
and similarly
\begin{align*}
\int_{-3}^3 \mu^2( S_z \cap E ) \mu^2(S_z)\ dz &= \int_E \int_{[-1,1]^2} \int_{-3}^3 1_{|x_1+x_4-z|, |x_2+x_3-z| \leq r_0}\ dz d\mu^2(x_2,x_3) d\mu^2(x_1,x_4) \\
&\lesssim r_0 \int_E \int_{[-1,1]} \mu(B(x_1-x_2+x_4,2r_0))\ d\mu(x_2) d\mu^2(x_1,x_4) \\
&\lesssim r_0 \int_E \mu([-1,1]) C r_0^\delta\ d\mu^2(x_1,x_4) \\
&\lesssim C^2 r_0^{1+\delta} \mu^2(E).
\end{align*}
Finally, for any $z$, one has from the Fubini--Tonelli theorem again that
\begin{align*}
\mu^2(X^2 \cap S_z \backslash E) &\leq \int_{\pi(X^2 \cap S_z \backslash E)} \mu([x-r_0,x+r_0])\ d\mu(x) \\
&\lesssim C r_0^\delta \mu( \pi(X^2 \cap S_z \backslash E) ).
\end{align*}
Combining all these estimates, we obtain the claim.
\end{proof}

To construct this exceptional set $E$ we will take advantage of the porous nature of the set $X$ as viewed through a sparse $K$-adic mesh.  Let $K \geq 1000$ be a perfect square and multiple of $100$ to be chosen later, chosen so large that
\begin{equation}\label{k1d}
K^{\delta/2}, K^{(1-\delta)/2} \geq C_0 C^2
\end{equation}
for some large absolute constant $C$.  Define a \emph{$K$-adic interval} to be a half-open interval of the form $I = [jK^{-n}, (j+1)K^{-n})$ for some integers $j,n$.  Such an interval will be said to be \emph{active} if it intersects $X$, and \emph{inactive} otherwise.  We have the following basic porosity property:

\begin{lemma}[Porosity]\label{poros}  There does not exist a sequence $I_1,\dots,I_{\sqrt{K}}$ of $\sqrt{K}$ consecutive active $K$-adic intervals of equal length in $[r_0,1]$.
\end{lemma}

\begin{proof}  Suppose for contradiction that there was a sequence $I_1,\dots,I_{\sqrt{K}}$ of consecutive active $K$-adic intervals of some length $K^{-n} \in [r_0,1]$.  From the regularity of $\mu$, we have
$$ \mu( I_j + [-K^{-n}, K^{-n}]) \gtrsim C^{-1} K^{-n\delta}$$
for all $j=1,\dots,\sqrt{K}$, hence on summing and noting the bounded overlap of the intervals $I_j + [-K^{-n}, K^{-n}]$
$$ \mu\left( \bigcup_{j=1}^{\sqrt{K}} I_j + [-K^{-n}, K^{-n}]\right) \gtrsim C^{-1} \sqrt{K} K^{-n\delta}.$$
On the other hand, from regularity again we have
$$ \mu\left( \bigcup_{j=1}^{\sqrt{K}} I_j + [-K^{-n}, K^{-n}]\right) \lesssim C K^{(1/2-n)\delta}.$$
This contradicts \eqref{k1d} if $C_0$ is large enough.
\end{proof}

Now we make some further definitions concerning $K$-adic intervals (see Figure \ref{fig:family}):

\begin{itemize}
\item[(i)] If $I = [jK^{-n}, (j+1)K^{-n})$ is a $K$-adic interval, we define the \emph{right shift} of $I$ to be the $K$-adic interval $I^+ \coloneqq [(j+1)K^{-n}, (j+2)K^{-n})$, the \emph{left shift} of $I$ to be the $K$-adic interval $I^- \coloneqq [(j-1)K^{-n}, jK^{-n})$, and the \emph{parent} $I^*$ to be the unique $K$-adic interval of $I$ of length $K^{1-n}$ that contains $I$. We then call $I$ a \emph{child} of $I^*$, and define a \emph{grandchild} to be a child of a child.  
\item[(ii)] Given two intervals $I,J$, we say that $I$ \emph{lies to the left} of $J$, or equivalently that $J$ \emph{lies to the right} of $I$, if $x < y$ for all $x \in I$ and $y \in J$.
\item[(iii)] A \emph{sibling} of a $K$-adic interval $I$ is another $k$-adic interval $J$ with the same parent as $I$: $I^* = J^*$.  A \emph{left sibling} (resp. right sibling) of $I$ is a sibling $J$ that lies to the left (resp. right) of $I$.
\item[(iv)] A \emph{left edge} is an active $K$-adic interval $I$ of length $K^{-2n} \in [r_0,1]$ for some $n \geq 1$, with the property that all left siblings of $I$, as well as the left shift $(I^*)^-$ of $I$, are inactive.  
\item[(v)]  A \emph{left near-edge} is a $K$-adic interval $J$, to which there is associated a left edge $I$ of the same length as $J$ and equal to or to the left of $J$, such that all the $K$-adic intervals of the same length as $J$ between $J$ and $I$ are active.  (In particular, if $I$ is a left edge, then $I$ and $I^+$ are left near-edges, and from Lemma \ref{poros} each left edge is associated to at most $K/100$ left near-edges, which are all adjacent to each other and of the same length as the left edge, with the rightmost of these left near-edges being inactive.
\end{itemize}

\begin{figure} [t]
\centering
\includegraphics[width=4in]{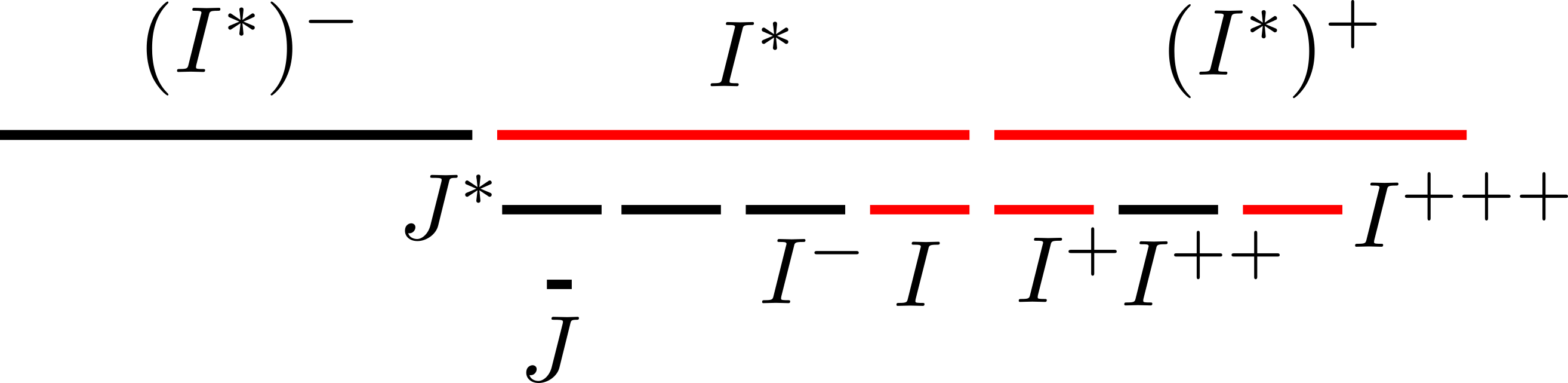}
\caption{Some selected $K$-adic intervals (with $K=4$).  The interval $J$ here is a grandchild of $I^*$, and its parent $J^*$ is a left sibling of $I$.  If the intervals in red are active and the intervals in black are inactive, and $I$ has length $K^{-2n} \in [r_0,1]$ for some $n \geq 1$, then $I$ will be a left edge and $I, I^+, I^{++}$ will be left near-edges, but none of the other intervals of length $K^{-2n}$ depicted here will be left near-edges.}
\label{fig:family}
\end{figure}

One could also define the notion of a right edge and right near-edge, but we will not need to do so here.

\begin{example} Let $K=10$, and let $X \subset [0,1]$ be the set of all real numbers whose decimal expansions take values in the set $\{ 1, 2, 5, 6\}$.  Then an $10$-adic interval of length $10^{-2n}$ is an interval of the form $[\frac{a}{10^{2n}}, \frac{a+1}{10^{2n}})$, where $a$ is an integer.  This interval is active if $a$ is positive with decimal expansion length $2n$ and has all digits in $\{1,2,5,6\}$; a left-edge arises if furthermore the final digit is $1$ and the penultimate digit lies in $\{1,5\}$.  Finally, a left near-edge of length $10^{-2n}$ arises if $a$ is positive with decimal expansion of length $2n$ with the final digit in $\{1,2,3\}$, penultimate digit in $\{1,5\}$, and all other digits in $\{1,2,5,6\}$.  Observe that almost all elements of $X$ will lie in at least one left near-edge, in the sense that the set of exceptions has strictly smaller dimension than $X$ itself.  This phenomenon of abundance of left near-edges is crucial to our argument.
\end{example}

We make the following basic observations.  Firstly, if $I$ is a left edge, then it is active, and from the regularity property of $\mu$ we have
$$ \mu( I^- \cup I \cup I^+ ) \gtrsim C^{-1} |I|^\delta.$$
On the other hand, as $I$ is a left edge, $I^-$ cannot be active (it is either a left sibling of $I$, or lies in $(I^*)^-$) and thus
\begin{equation}\label{muii}
 \mu( I \cup I^+ ) \gtrsim C^{-1} |I|^\delta.
 \end{equation}
Thus at least one of the left near-edges $I,I^+$ associated to $I$ must absorb a relatively large amount of the mass of $\mu$ (compared to the upper bound of $C |I|^\delta$ coming from the regularity hypothesis).

Next, we claim that any two left edges $I,J$ of the same length $|I|=|J|$ must be separated from each other by at least $K|I|$.  Indeed, we may assume without loss of generality that $J$ lies to the left of $I$; as the left siblings of $I$, as well as the entirety of $(I^*)^-$, are inactive, this forces $J$ to lie to the left of $(I^*)^-$, giving the claim.  In particular, we see that the left near-edges associated to $I$ are disjoint from the left near-edges associated to $J$.

Now we make the key claim that lets us produce many left edges at many scales.

\begin{lemma}[Many left edges]\label{many-left}  Let $n \geq 1$ be such that $K^{-2n} \in [K^2 r_0, 1]$.
 Let $I$ be an active $K$-adic interval of length $K^{-2n}$.  Then $I \cup I^-$ contains a left edge $J$ of length $K^{-2} |I| = K^{-2(n+1)}$.  Furthermore, if $J$ is associated to a left near-edge $J'$ that is contained in a larger left near-edge $\tilde J$, then $I$ is also contained in a left-near edge of the same length as $\tilde J$.
\end{lemma}

\begin{figure} [t]
\centering
\includegraphics[width=4in]{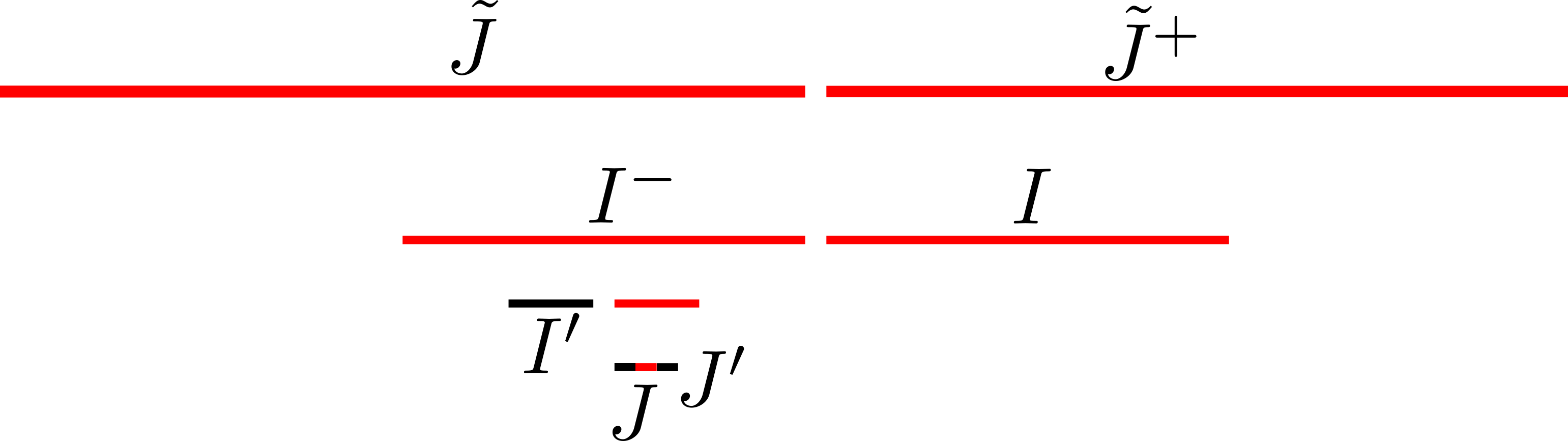}
\caption{A typical situation that arises in the proof of Lemma \ref{many-left} (only a portion of the large intervals $\tilde J, \tilde J^+$ are depicted here).  Active intervals are displayed in red, inactive intervals in black. In this example $J, J'$ is contained in $I^-$, but it is also possible for one or both of these intervals to lie in $I$ instead.  Similarly, it is also possible for $I$ to lie in $\tilde J$ rather than $\tilde J^+$. Note that $J$ is a left edge.}
\label{fig:manyleft}
\end{figure}

\begin{proof}  By Lemma \ref{poros}, at least one of the children $I'$ of $I^-$ will be inactive.  On the other hand, by the pigeonhole principle at least one of the grandchildren of $I$ is active.  In particular we can find an active interval $J$ of length $K^{-2} |I|$ that lies to the right of $I'$, is a grandchild of either $I$ or $I^-$, and is the leftmost interval with these properties; see Figure \ref{fig:manyleft}.  By Lemma \ref{poros} $J$ lies at a distance of at most $\sqrt{K} |I'| = K^{-1/2} |I'|$ of $I'$, and thus is a grandchild of either $I$ or $I'$.

By construction, all the intervals of length $K^{-2} |I|$ between $I'$ and $J$ are inactive.  As $I'$ is also inactive, this makes $J$ a left edge by definition.  Now suppose that there is a left near-edge $J'$ associated to $J$ is contained in a larger left near-edge $\tilde J$, thus $\tilde J$ is at least as large as $I$.  Note that $J'$ is either equal to $J$, or lies to the right at a distance of at most $\sqrt{K} |J| = K^{-3/2} |I|$ thanks to Lemma \ref{poros}; in particular, $J'$ also lies in $I^- \cup I$.  If $J'$ lies in $I$ then $\tilde J$ will contain $I$ and we are done, so suppose $J'$ lies in $I^-$.  Then $\tilde J$ contains $I^-$, hence contain $J$, hence is active, hence $\tilde J^+$ is also a left near-edge. As $\tilde J$ contains $I^-$, $I$ will be contained in either $\tilde J$ or $\tilde J^+$, and the claim follows.
\end{proof}

Now we can construct the exceptional set $E \subset [-1,1]^2$.  Let $N$ be the largest integer such that $K^{-2N} \geq r_0$.  For any $0 \leq n \leq N$, we let $E_n$ be the set of all elements of $[-1,1]^2$ that do not lie in a square $I_1 \times I_2$, where $I_1,I_2$ are two left near-edges of equal length $K^{-2n'}$ for some $1 \leq n' \leq n$.  Clearly $E_0 = [-1,1]^2$, hence by regularity
$$ \mu^2(E_0) \lesssim C^2.$$

Thanks to Lemma \ref{many-left}, we can now show a geometric decrease in the measures of the $E_n$:

\begin{proposition}[Geometric decrease]  For every $0 \leq n \leq N-1$, one has
$$ \mu^2(E_{n+1}) \leq (1-c C^{-4} K^{-4\delta}) \mu^2(E_{n})$$
for some absolute constant $c>0$.
\end{proposition}

\begin{proof}  It suffices to show that
$$ \mu^2(E_n \backslash E_{n+1}) \gtrsim C^{-4} K^{-4\delta} \mu^2(E_n).$$
By construction, $E_n$ is the union of squares $I_1 \times I_2$, where $I_1,I_2$ are $K$-adic intervals of length $K^{-2n}$ such that there is no $1 \leq n' \leq n$ for which $I_1, I_2$ are respectively contained in two left-near edges $\tilde I_1, \tilde I_2$ of length $K^{-2n'}$.  If $I_1$ or $I_2$ are inactive then the square $I_1 \times I_2$ has zero measure in $\mu$, so we can restrict attention to squares $I_1 \times I_2$ which are \emph{active} in the sense that $I_1$ and $I_2$ are both active.  By regularity, the contribution of each active square can be bounded by
$$ \mu^2(I_1 \times I_2) \lesssim (C |I_1|^\delta) (C |I_2|^\delta) = C^2 K^{-4n\delta}.$$
Thus, if there are $M$ active squares, we have
$$ \mu^2(E_n) \lesssim C^2 K^{-4n\delta} M$$
and so it will now suffice to establish the bound
$$ \mu^2(E_n \backslash E_{n+1}) \gtrsim  C^{-2} K^{-4n\delta-4\delta} M.$$
From Lemma \ref{many-left}, for each active square $I_1 \times I_2$ one can find left edges $J_1 \subset I_1^- \cup I_1$, $J_2 \subset I_2^- \cup I_2$ of length $K^{-2n-2}$ obeying the conclusions of the lemma.  From \eqref{muii} and the pigeonhole principle we can find for each $i=1,2$ an interval $J'_i$ that is either equal to $J_i$ or its right shift $J_i^+$, such that
$$ \mu(J'_i) \gtrsim C^{-1} (K^{-2n-2})^\delta$$
and hence
$$ \mu^2(J'_1 \times J'_2) \gtrsim C^{-2} K^{-4n\delta-4\delta}.$$
Note that $J'_1,J'_2$ are left near-edges associated to $J_1,J_2$ respectively.  In particular, from Lemma \ref{many-left}, since $I_1,I_2$ fail to be respectively contained in left-near edges of length $K^{-2n'}$ for some $1 \leq n' \leq n$, the same is true for $J'_1, J'_2$.  By construction of $E_{n+1}$, this implies that
$$ J'_1 \times J'_2 \subset E_n \backslash E_{n+1}.$$
Since $J'_1 \times J'_2$ lies within $O(K^{-2n})$ of $I_1 \times I_2$, we see that each square $J'_1 \times J'_2$ can be generated by at most $O(1)$ of the $M$ active squares $I_1 \times I_2$.  Thus we have
$$ \mu^2( E_n \backslash E_{n+1} ) \gtrsim C^{-2} K^{-4n\delta-4\delta} M$$
giving the claim.
\end{proof}

Iterating this proposition we see that
\begin{align*}
\mu^2(E) &\lesssim C^2 (1-c C^{-4} K^{-4\delta})^N \\
&\lesssim C^2 \exp( -c N C^{-4} K^{-4\delta} ) \\
&\lesssim C^2 r_0^{\frac{c'}{C^4 K^{4\delta} \log K}}
\end{align*}
for some absolute constants $c, c'>0$.

Now let $z \in [-3,3]$.  In view of Corollary \ref{conc}, we are interested in bounding the expression
$\mu(\pi(X^2 \cap S_z \backslash E)) )$.  We can of course write this as $\mu(Z)$, where $Z$ is the set
$$ Z \coloneqq \pi(X^2 \cap S_z \backslash E).$$
By construction of $E$, we can then bound
$$ \mu(Z) \leq \sum_{I_1,I_2} \mu\left( \pi\left(X^2 \cap S_z \cap \left(\bigcup_{I'_1} I'_1 \times \bigcup_{I'_2} I'_2\right)\right) \right)$$
where $I_1,I_2$ range over pairs of left-edges of equal length in $[r_0,1]$, and $I'_1,I'_2$ range over the left near-edges associated with $I_1,I_2$ respectively.  Now we make a key calculation.

\begin{lemma}[Bounding a piece of $\mu(Z)$]\label{muz-piece}  To every pair $I_1,I_2$ of left edges of equal length in $[r_0,1]$, one can find a set $Y_{I_1,I_2} \subset [-2,2]$ such that
\begin{equation}\label{yip}
 \mu\left( \pi\left( X^2 \cap S_z \cap \left(\bigcup_{I'_1} I'_1 \times \bigcup_{I'_2} I'_2\right) \right) \right) \lesssim C^2 K^{-\delta/2} \mu(Y_{I_1,I_2}),
 \end{equation}
where $I'_1, I'_2$ range over the left near-edges associated with $I_1, I_2$ respectively.
Furthermore, the sets $Y_{I_1,I_2}$ are disjoint as $(I_1,I_2)$ vary.
\end{lemma}

The key points here are the gain of $K^{-\delta/2}$ on the right-hand side, and the disjointness of the sets $Y_{I_1,I_2}$.

\begin{figure} [t]
\centering
\includegraphics[width=4in]{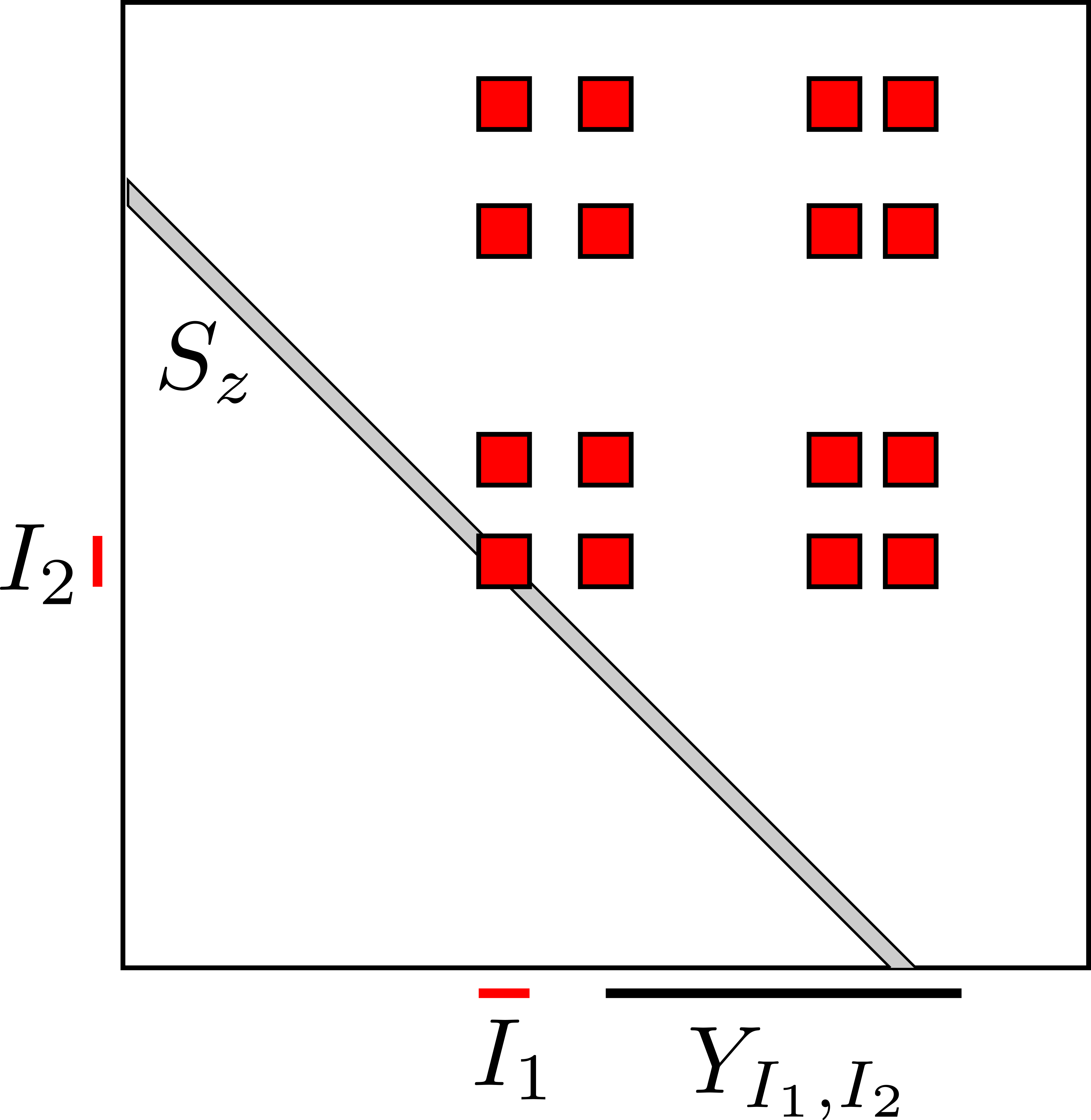}
\caption{A typical situation that arises in the proof of Lemma \ref{muz-piece}.  To reduce clutter we take $I_1=I'_1$ and $I_2=I'_2$.  The set $X^2$ (which supports the measure $\mu^2$) is contained in the squares formed by products of active intervals, such as $I_1 \times I_2$ (which is a typical component of the complement of $E$).  Because $I_1,I_2$ are left-edges, if $S_z$ intersects $I'_1 \times I'_2 = I_1 \times I_2$, then the set $\pi(X^2 \cap S_z \cap (I'_1 \times I'_2))$, being contained in $I_1 = I'_1$, will be much smaller than the nearby $Y_{I_1,I_2}$ as measured using the regular measure $\mu$; on the other hand, $Y_{I_1,I_2}$ stays at a medium distance from the set $Z = \pi(X^2 \cap S_z \backslash E)$ and from $I_1$, which will ensure that the $Y_{I_1,I_2}$ are disjoint as $I_1,I_2$ vary.  Note how this picture resembles the Cartesian product of the two portions of the real line depicted in Figure \ref{fig:leftedge}.}
\label{fig:gain}
\end{figure}

\begin{proof}  Write $I_1 = [x_1, x_1 + K^{-2n})$ and $I_2 = [x_2, x_2 + K^{-2n}]$ for some $K^{-2n} \in [r_0,1]$.  We can assume that the set $X^2 \cap S_z \cap (\bigcup_{I'_1} I'_1 \times \bigcup_{I'_2} I'_2)$ is non-empty, otherwise we can simply set $Y_{I_1,I_2}$ to be the empty set.  From Lemma \ref{poros}, the intervals $I'_i$ lie within $2 \sqrt{K} K^{-2n}$ of $x_i$ for $i=1,2$, hence by definition of $S_z$ we have
\begin{equation}\label{zoo}
|z - x_1 - x_2| \leq 5 \sqrt{K} K^{-2n}.
\end{equation}
also we see that from the definition of $Z$ that $Z$ contains a point within $2 \sqrt{K} K^{-2n}$  of $x_1$, thus
\begin{equation}\label{disp}
\mathrm{dist}(x_1,Z) \leq 2\sqrt{K} K^{-2n}.
\end{equation}
We define $Y_{I_1,I_2}$ to be the interval
\begin{equation}\label{y-def}
 Y_{I_1,I_2} \coloneqq \left[x_1 + 100 \sqrt{K} K^{-2n}, x_1 + \frac{1}{100} K^{-2n+1}\right];
\end{equation}
see Figure \ref{fig:gain}.  Clearly $Y_{I_1,I_2} \subset [-2,2]$.  We now verify \eqref{yip}.
We can use regularity to bound
\begin{align*}
\mu\left( \pi\left( S_z \cap \left(\bigcup_{I'_1} I'_1 \times \bigcup_{I'_2} I'_2\right) \right) \right) 
&\leq \mu\left( \bigcup_{I'_1} I'_1 \right)\\
&\lesssim C (\sqrt{K} K^{-2n})^\delta.
\end{align*}
Next, as $I_1$ is a left edge, it contains an element $x_*$ of $X$, and
$$ \mu([x_1 - K^{-2n+1}, x_1)) = 0$$
and hence
$$ \mu(Y_{I_1,I_2}) \geq \mu\left( B(x_*,\frac{1}{200} K^{-2n+1}) \backslash B(x_*,200 \sqrt{K} K^{-2n}) \right).$$
From regularity we have
$$\mu\left( B(x_*,\frac{1}{200} K^{-2n+1}) \right) \gtrsim C^{-1} K^{(-2n+1)\delta}$$
and
$$ \mu\left( B(x_*,200 \sqrt{K} K^{-2n}) \right) \lesssim C K^{(-2n+1/2)\delta}$$
hence by \eqref{k1d} we have (for $C_0$ large enough) that
$$ \mu(Y_{I_1,I_2}) \gtrsim C^{-1} K^{(-2n+1)\delta}$$
and \eqref{yip} follows.

It remains to establish the disjointness of the $Y_{I_1,I_2}$.  To do this we study the relative position of $Y_{I_1,I_2}$ and $Z$.  From \eqref{disp} one has
\begin{equation}\label{y1}
\mathrm{dist}(y,Z) \leq \frac{1}{10} K^{-2n+1}
\end{equation}
for all $y \in Y_{I_1,I_2}$.  Now we use the fact that $I_2$ is a left-edge, thus
$$ X \cap [x_2 - K^{-2n+1}, x_2) = \emptyset.$$
By definition of $S_z$, this implies that the set $Z \subset \pi( X^2 \cap S_z )$ avoids the interval $(z - x_2+r_0, z-x_2+K^{-2n+1}-r_0)$.  Using \eqref{zoo}, we conclude that $Z$ avoids the interval
$$  \left[x_1 + 10 \sqrt{K} K^{-2n}, x_1 + \frac{1}{2} K^{-2n+1}\right]$$
(say), and thus we have
\begin{equation}\label{y2}
\mathrm{dist}(y,Z) \geq 10 K^{-2n+1/2}
\end{equation}
for all $y \in Y_{I_1,I_2}$.

The bounds \eqref{y1}, \eqref{y2} imply that two intervals $Y_{I_1,I_2}, Y_{I'_1,I'_2}$ cannot overlap if the lengths $|I_1|=|I_2|=K^{-2n}$, $|I'_1|=|I'_2|=K^{-2n'}$ are distinct.  Thus it remains to show that the intervals $Y_{I_1,I_2}, Y_{I'_1,I'_2}$ cannot overlap in the case of all equal lengths $|I_1|=|I_2|=|I'_1|=|I'_2|=K^{-2n}$.  Recall that if the left-edges $I_1,I'_1$ are distinct, they are separated from each other by at least $K^{-2n+1}$.  Comparing this with \eqref{y-def} we see that one can only have $Y_{I_1,I_2}, Y_{I'_1,I'_2}$ intersect if $I_1=I'_1$.  But then from \eqref{zoo} we see that left-edges $I_2, I'_2$ are separated by at most $10 \sqrt{K} K^{-2n}$, and hence must also be equal from the separation property of left-edges.  Thus we see that the $Y_{I_1,I_2}$ are indeed disjoint as $(I_1,I_2)$ vary.
\end{proof}

Summing the above lemma over all pairs $(I_1,I_2)$ we conclude that
$$ \mu(Z) \lesssim C^2 K^{-\delta/2} \mu([-2,2]) \lesssim C^3 K^{-\delta/2}$$
and hence by Corollary \ref{conc} we have
$$ \Energy(\mu|_{[-1,1]}, r_0) \lesssim C^6 r_0^\delta \left( r_0^{\frac{c'}{C^4 K^{4\delta} \log K}} + K^{-\delta/2} \right).$$
If we now set
$$ K \coloneqq (C_1 C^6/\eps)^{\max(\frac{2}{\delta}, \frac{2}{1-\delta})}$$
for a sufficiently large absolute constant $C_1$, the the condition \eqref{k1d} will be satisfied, and
$$ \Energy(\mu|_{[-1,1]}, r_0) \leq r_0^\delta \left(\frac{\eps}{2} + O\left( C^6 r_0^{\frac{c'' \eps^2 \min(\delta,1-\delta)}{C^{16} \log(C/\eps)}} \right) \right)$$
for some absolute constant $c>0$.  Thus if we select
$$ r_0 \coloneqq \exp\left( - C_2 \frac{C^{16} \log^2(C/\eps)}{\eps^2 \min(\delta,1-\delta)} \right)$$
for a sufficiently large absolute constant $C_2$, we obtain
$$ \Energy(\mu|_{[-1,1]}, r_0) \leq \eps r_0^\delta$$
as required.  This concludes the proof of Proposition \ref{slight-gain-1}.

\section{The one-dimensional case: induction on scales}\label{induct-sec}

In this section we show how one can iterate Proposition \ref{slight-gain-1} to obtain Theorem \ref{main}.

Let $\delta, C, \alpha, X, \mu_X$ be as in Theorem \ref{main}.  Let $\eps>0$ be a small parameter to be chosen later, and let $r_0$ be the quantity in Proposition \ref{slight-gain-1}.

It is convenient to adopt the notation of Gowers uniformity norms \cite{gowers}.  For $f_1,f_2,f_3,f_4 \in L^{4/3}(\R)$, define the Gowers inner product
$$ \langle f_1,f_2,f_3,f_4 \rangle_{U^2(\R)} \coloneqq \int_\R \int_\R \int_\R f_1(x) \overline{f_2}(x+h) \overline{f_3}(x+k) f_4(x+h+k)\ dx dh dk$$
and the Gowers uniformity norm
$$ \|f\|_{U^2(\R)} \coloneqq \langle f,f,f,f \rangle^{1/4}$$
for any $f \in L^{4/3}(\R)$.  One can also write the $U^2$ norm in terms of the Fourier transform as
\begin{equation}\label{u24}
\|f\|_{U^2(\R)} = \| \hat f \|_{L^4(\R)}
\end{equation}
so it is clear that the $U^2(\R)$ norm is indeed a norm.  We also recall the well-known \emph{Gowers--Cauchy--Schwarz inequality}
\begin{equation}\label{gcz-1}
|\langle f_1,f_2,f_3,f_4 \rangle_{U^2(\R)}| \leq \|f_1\|_{U^2(\R)} \|f_2\|_{U^2(\R)} \|f_3\|_{U^2(\R)} \|f_4\|_{U^2(\R)}
\end{equation}

We can relate the Gowers norms to additive energy as follows:

\begin{lemma}[Additive energy and Gowers norms]\label{agn}  If $\mu$ is a finite Radon measure on $\R$ and $r>0$, then
\begin{equation}\label{emr}
 \Energy(\mu,r) \sim r^{-3} \| \mu * 1_{[-r,r]} \|_{U^2}^4
 \end{equation}
 (compare with Lemma \ref{mor}).
Also, for any $\lambda>0$ one has
\begin{equation}\label{lamr}
\| \mu * 1_{[-\lambda r,\lambda r]} \|_{U^2} \sim_\lambda 
 \| \mu * 1_{[-r,r]} \|_{U^2}.
 \end{equation}
 \end{lemma}

\begin{proof}  The claim \eqref{lamr} is clear after observing that $1_{[-\lambda r, \lambda_r]}$ can be bounded by the sum of $O_\lambda(1)$ translates of $1_{[-r,r]}$ (and vice versa), together with the triangle inequality and translation invariance of Gowers norms.

By the Fubini--Tonelli theorem we have
$$
\| \mu * 1_{[-r,r]} \|_{U^2}^4
= \int_{\R^4} 1_{[-r,r]} * 1_{[-r,r]} * 1_{[-r,r]} * 1_{[-r,r]}(x_1+x_2-x_3-x_4)\ d\mu(x_1) \dots d\mu(x_4).$$
Since
$$ 1_{[-r,r]} * 1_{[-r,r]} * 1_{[-r,r]} * 1_{[-r,r]} \gtrsim r^3 1_{[-r,r]}$$
we obtain the bound
$$ \| \mu * 1_{[-r,r]} \|_{U^2}^4 \gtrsim r^3 \Energy(\mu,r).$$
Conversely, since
$$ 1_{[-r/4,r/4]} * 1_{[-r/4,r/4]} * 1_{[-r/4,r/4]} * 1_{[-r/4,r/4]} \lesssim r^3 1_{[-r,r]}$$
we have
$$ \| \mu * 1_{[-r/4,r/4]} \|_{U^2}^4 \lesssim r^3 \Energy(\mu,r)$$
and the claim \eqref{emr} now follows from \eqref{lamr}.
\end{proof}

The key claim is the following inequality relating the energy at different scales:

\begin{proposition}[Energy at nearby scales]\label{en-nearby}  If $\alpha \leq r \leq r_0$, one has
$$ \Energy(\mu,r) \lesssim C^4 \eps r_0^\delta \Energy(\mu,r/r_0)$$
\end{proposition}

\begin{proof}  By Lemma \ref{agn} one has
$$ \Energy(\mu,r) \sim r^{-3} \| \mu * 1_{[-r,r]} \|_{U^2(\R)}^4.$$
Partition $\R$ into a collection ${\mathcal I}$ of half-open intervals $I$ of length $r/r_0$, then we can decompose
$$ \mu = \sum_I \mu|_I$$
where $\mu_I$ denotes the restriction of $\mu$ to $I$, and thus
$$ \Energy(\mu,r) \sim \sum_{I_1,I_2,I_3,I_4} 
r^{-3} \langle \mu|_{I_1} * 1_{[-r,r]}, \dots, \mu|_{I_4} * 1_{[-r,r]} \rangle_{U^2(\R)}$$
where $I_1,\dots,I_4$ are understood to vary over ${\mathcal I}$.  The integral vanishes unless $I_1-I_2-I_3+I_4$ intersects $[-4r,4r]$ and all of the $\mu|_{I_i}$ are non-vanishing, which by Definition \ref{reg-set}(iii) implies that
$$ \mu(2I_i-I_i) \gtrsim C^{-1} (r/r_0)^\delta,$$
where $2I_i-I_i = I_i + [-r/r_0,r/r_0]$ is the triple of $I_i$.  When this occurs, we can use the Gowers--Cauchy--Schwarz inequality to estimate
$$
\langle \mu|_{I_1} * 1_{[-r,r]}, \dots, \mu|_{I_4} * 1_{[-r,r]} \rangle_{U^2(\R)}
\leq \prod_{i=1}^4 \| \mu|_{I_i} * 1_{[-r,r]}\|_{U^2(\R)}$$
and hence by Lemma \ref{agn} again
$$
r^{-3} \langle \mu|_{I_1} * 1_{[-r,r]}, \dots, \mu|_{I_4} * 1_{[-r,r]} \rangle_{U^2(\R)}
\lesssim \prod_{i=1}^4 \Energy(\mu|_{I_i}, r)^{1/4}.$$

For each $i$, let $T_i \colon \R \to \R$ be the affine (order-preserving) map that sends $I_i$ to $[-1,1]$.  Direct calculations shows that the pushforward measure $(r_0/r)^\delta T_* \mu$ is $\delta$-regular at scales $[\alpha r_0/r, r_0/r]$, and in particular at scales $[r_0,1]$.  Applying Proposition \ref{slight-gain-1} to this measure and undoing the rescaling, we see after a routine calculation that
$$ \Energy(\mu|_{I_i}, r_0) \leq \eps (r/r_0)^{4\delta} r_0^\delta \lesssim C^4 \eps r_0^\delta \mu(2I_i-I_i)^4.$$
Putting all this together, we see that
$$ \Energy(\mu,r) \lesssim C^4 \eps r_0^\delta
 \sum_{I_1,I_2,I_3,I_4: I_1-I_2-I_3+I_4 \cap [-4r,4r] \neq \emptyset} \prod_{i=1}^4 \mu(2I_i-I_i).$$
By the Fubini--Tonelli theorem, we can write the right-hand side as
$$
C^4 \eps r_0^\delta \int_{\R^4}
 \sum_{I_1,I_2,I_3,I_4: I_1-I_2-I_3+I_4 \cap [-4r,4r] \neq \emptyset} \prod_{i=1}^4 1_{2I_i-I_i}(x_i)\ d\mu(x_1) \dots d\mu(x_4).$$
The integrand can be computed to equal $O(1)$, and vanishes unless
$$ |x_1 - x_2 - x_3 + x_4| \leq 100 r/r_0$$
(say).  Thus we conclude that
$$ \Energy(\mu,r) \lesssim C^4 \eps r_0^\delta \Energy(\mu,100r/r_0)$$
and the claim now follows by using Lemma \ref{agn} to remove the factor of $100$.
\end{proof}

If we now set $\eps$ to be a small multiple of $1/C^4$, we can ensure that
$$ \Energy(\mu,r) \leq \frac{1}{e} r_0^\delta \Energy(\mu,r/r_0)$$
and hence upon induction (starting with the base case $r_0 \leq r \leq 1$) one has
$$ \Energy(\mu,r) \lesssim (r/r_0)^{\delta + \frac{1}{\log r_0}}  \Energy(\mu,1)$$
for all $\alpha \leq r \leq 1$.  Substituting the specific value \eqref{r0-def} corresponding to the indicated choice of $\eps$, we obtain the claim.

\section{Gowers norms and approximate groups}\label{add-comb}

In the previous section we connected the additive energy $\Energy(\mu,r)$ to the Gowers uniformity norms used in additive combinatorics.  We now develop the theory of these norms (and the related notion of an \emph{approximate group}) in more detail, as we shall rely heavily on these tools for the higher-dimensional argument.  As we will be working in the continuous setting of Euclidean spaces $\R^d$ rather than in the discrete settings that are more traditional in additive combinatorics, we shall phrase these concepts in the general setting of locally compact abelian groups. In this section we will lay out the basic theory of these concepts; most of the material is standard, except for a key ``scale splitting'' estimate which will underlie various manifestations of an ``induction on scales'' strategy; see Lemma \ref{split} and Lemma \ref{relate}(iii).

\begin{definition}[LCA groups]\label{lca-group} An \emph{LCA group} is a locally compact abelian group $V = (V,+)$ equipped with a Haar measure $m_V$, and the Borel sigma algebra.  For any $1 \leq p \leq \infty$, we define the usual Banach spaces $L^p(V)$ of $p^{\mathrm{th}}$ power integrable functions $f \colon V \to \C$, quotiented out by almost everywhere equivalence. We let $L^p(V)_+$ denote the subset of $L^p(V)$ consisting of functions that are non-negative (almost everywhere).  For any positive measure subset $X$ of $V$, we define $L^p(X)$ and $L^p(X)_+$ similarly.
\end{definition}

In this paper we will mostly be concerned with the case when $V$ is a Euclidean space $\R^d$ equipped with Lebesgue measure, though we will occasionally also need to work with the lattice $\Z^d$ with counting measure, or hyperplanes $v^\perp \coloneqq \{ x \in \R^d: x \cdot v = 0 \}$ equipped with Lebesgue measure.

\begin{definition}[Gowers uniformity norm]\cite[Definition 1.1]{eisner}  Let $V$ be an LCA group.  If $f_1,f_2,f_3,f_4 \in L^{4/3}(V)$, we define the Gowers inner product
$$ \langle f_1,f_2,f_3,f_4 \rangle_{U^2(V)} \coloneqq \int_V \int_V \int_V f_1(x) \overline{f_2}(x+h) \overline{f_3}(x+k) f_4(x+h+k)\ dm_V(x) dm_V(h) dm_V(k)$$
(these integrals can be shown to be absolutely convergent by the H\"older and Young inequalities, or by interpolation) and for $f \in L^{4/3}(V)$ we define the uniformity norm $\|f\|_{U^2(V)} = \|f\|_{U^2(V,)}$ by the formula
$$ \|f\|_{U^2(V)} \coloneqq \langle f,f,f,f \rangle_{U^2(V)}^{1/4}.$$
If $\mu$ is an absolutely continuous Radon measure on $V$ whose Radon-Nikodym derivative $\frac{d\mu}{dm_V}$ is finite, we define (by abuse of notation)
$$ \| \mu \|_{U^2(V)} \coloneqq \left\| \frac{d\mu}{dm_V} \right\|_{U^2(V)}.$$
Finally, if $X$ is a positive measure subset of $V$ and $f \in L^{4/3}(X)$, we define
$$ \|f \|_{U^2(X)} \coloneqq \|f 1_X \|_{U^2(V)}.$$
\end{definition}

Although we define the Gowers norms here for $f \in L^{4/3}(V)$, for our applications it would suffice to restrict attention to the non-negative functions $f \in L^{4/3}(V)_+$.

As noted in \cite{eisner}, the norms $\|\|_{U^2(V)}$ are indeed norms; in fact one has an explicit representation
\begin{equation}\label{u24-v}
\|f\|_{U^2(V)} = \|f*f\|_{L^2(V)}^{1/2} = \| \hat f \|_{L^4(\hat V)}
\end{equation}
in terms of the $L^4$ norm of the Fourier transform $\hat f \colon \hat V \to \C$ on the Pontragyin dual $\hat V$ of $V$ (equipped with the dual Haar measure $m_{\hat V}$), where we use the usual convolution operation
$$ f*g(x) \coloneqq \int_V f(x-y) g(y)\ dm_V(y).$$
Similarly one has
$$ \langle f_1,f_2,f_3,f_4 \rangle_{U^2(V)} = \int_{\hat V}
\hat f_1(\xi) \overline{\hat f_2(\xi)} \overline{\hat f_3(\xi)} \hat f_4(\xi)\ dm_{\hat V}(\xi).$$
From Young's inequality or the Hausdorff-Young inequality one then has the bound
\begin{equation}\label{young}
\|f\|_{U^2(V)} \leq \|f\|_{L^{4/3}(V)}
\end{equation}
so in particular by H\"older's inequality
\begin{equation}\label{triv-u2}
 \|f\|_{U^2(V)} \leq \|f\|_{L^1(V)}^{3/4} \|f\|_{L^\infty(V)}^{1/4}
 \end{equation}
(compare with \eqref{triv-bound}).  From another application of H\"older (or Cauchy-Schwarz) we conclude the \emph{Gowers--Cauchy--Schwarz inequality}
\begin{equation}\label{gcz}
|\langle f_1,f_2,f_3,f_4\rangle_{U^2(V)}| \leq \|f_1\|_{U^2(V)} \|f_2\|_{U^2(V)} \|f_3\|_{U^2(V)} \|f_4\|_{U^2(V)}.
\end{equation}
For functions $f \in L^{4/3}(X)$ supported on a set $X$ of positive finite measure, we also have from Cauchy--Schwarz that
$$ \|f\|_{U^2(X)} = \|f*f \|_{L^2(2X)}^{1/2} \geq m_V(2X)^{-1/4} \|f*f \|_{L^1(2X)}^{1/2}$$
and hence
\begin{equation}\label{flower}
\|f\|_{U^2(X)} \geq \mu_V(2X)^{-1/4} \| f \|_{L^1(X)}.
\end{equation}

In view of \eqref{young}, one can think of the $U^2(V)$ norm of a function $f$ as the $L^{4/3}(V)$ norm multiplied by a dimensionless (scale-invariant) quantity that informally measures the amount of ``additive structure'' present on $V$.

The Gowers norm is clearly invariant under changes of variable by measure-preserving affine homomorphisms.  In particular it is translation-invariant, hence by Minkowski's inequality one has
$$ \|f*g\|_{U^2(V)} \leq \|f\|_{U^2(V)} \|g\|_{L^1(V)}$$
for all $f \in L^{4/3}(V)$ and $g \in L^1(V)$.  In a similar spirit, one has
\begin{equation}\label{fmuv}
 \|f*\mu\|_{U^2(V)} \leq \|f\|_{U^2(V)} 
 \end{equation}
for any $f \in L^{4/3}(V)$ and any Radon probability measure $\mu$, where the convolution is now defined as
$$ f*\mu(x) \coloneqq \int_V f(x-y)\ d\mu(y).$$

As mentioned previously, we will primarily be concerned with the Gowers norm on the non-negative cone $L^{4/3}(V)_+$ of $L^{4/3}(V)$.  Clearly the Gowers norm is monotone on this cone in the sense that
\begin{equation}\label{monotone}
 \|f\|_{U^2(V)} \leq \|g\|_{U^2(V)}
\end{equation}
whenever $f,g \in L^{4/3}(V)_+$ are such that $f \leq g$ pointwise almost everywhere.

If one has two functions $f \in L^{4/3}(V), f' \in L^{4/3}(V')$ on two LCA groups $V,V'$, then we can define the tensor product $f \otimes f' \in L^{4/3}(V \times V')$ on the product LCA group $V \times V'$ (equipped with product Haar measure) by the formula
$$ f \otimes f'(x,x') \coloneqq f(x) f'(x')$$
for $x \in V$, $x' \in V'$.  From the Fubini--Tonelli theorem we have the identity
\begin{equation}\label{fub-ton}
\| f \otimes f'\|_{U^2(V \times V')} = \|f\|_{U^2(V)} \|f'\|_{U^2(V')}.
\end{equation}
For functions that are not of tensor product form, we have the following ``splitting inequality'' that serves as a partial substitute for \eqref{fub-ton}:

\begin{lemma}[Splitting inequality]\label{split}  Let $V,V'$ be LCA groups.  If $f \in L^{4/3}(V \times V')$, then
\begin{equation}\label{uv}
 \|f\|_{U^2(V \times V')} \leq \|f_{V} \|_{U^2(V')}
 \end{equation}
where $f_{V} \in L^{4/3}(V')$ is the function
$$ f_{V}(v') \coloneqq \| f(\cdot,v') \|_{U^2(V)}$$
where $f(\cdot,v')$ is the function $v \mapsto f(v,v')$ (this is well-defined in $L^{4/3}(V)$ for almost every $v' \in V'$).
\end{lemma}

The right-hand side of \eqref{uv} can be thought of as an iterated norm $\|f\|_{U^2(V'; U^2(V))}$, so \eqref{uv} can be written as
$$ \|f\|_{U^2(V \times V')} \leq \|f \|_{U^2(V;U^2(V'))}$$
which can be compared with the Fubini--Tonelli identity
$$ \|f\|_{L^{4/3}(V \times V')} = \|f \|_{L^{4/3}(V;L^{4/3}(V'))}.$$
Thus one can view Lemma \ref{split} as an analogue of the Fubini--Tonelli theorem for the Gowers uniformity norm $U^2$.

\begin{proof}  The fact that $f_{V} \in L^{4/3}(V')$ follows from \eqref{young} and the Fubini--Tonelli theorem.  From another application of Fubini--Tonelli one has
\begin{align*}
\langle f,f,f,f \rangle_{U^2(V \times V')} 
&= \int_{V'} \int_{V'} \int_{V'} \langle f(\cdot,x'), f(\cdot,x'+h'), f(\cdot,x'+k'), f(\cdot,x'+h'+k') \rangle_{U^2(V)}\\
&\quad  dm_{V'}(x') dm_{V'}(h') dm_{V'}(k')
\end{align*}
and hence by the Gowers--Cauchy--Schwarz inequality \eqref{gcz}
\begin{align*} \langle f,f,f,f \rangle_{U^2(V \times V')} 
&\leq \int_{V'} \int_{V'} \int_{V'} f_V(x') f_V(x'+h') f_V(x'+k') f_V(x'+h'+k')\\
&\quad  dm_{V'}(x') dm_{V'}(h') dm_{V'}(k')
\end{align*}
or equivalently
$$ \|f\|_{U^2(V \times V')}^4 \leq \|f_V \|_{U^2(V')}^4$$
and the claim follows.
\end{proof}

We will be interested in \emph{inverse theorems} for the trivial inequality \eqref{triv-u2}, that is to say descriptions of those $f$ for which this inequality is close to sharp.  We now recall a key definition (see e.g., \cite[Definition 2.25]{tao-vu}:

\begin{definition}[Approximate group]  Let $V$ be an LCA group and $K \geq 1$.  A subset $H$ of $V$ is said to be a \emph{$K$-approximate group} if $H$ contains the origin, is symmetric (thus $-H=H$), and if $H+H$ can be covered by at most $K$ translates of $H$.  (In particular, this implies that $m_V(mH) \leq K^{m-1} m_V(H)$ for all natural numbers $m$.)  If $H$ has positive finite measure, we define the uniform probability measure $\nu_H$ on $H$ by the formula
$$ \nu_H(E) \coloneqq \frac{m_V(E \cap H)}{m_V(H)}$$
for all measurable $E \subset V$.  Similarly, if $H$ has finite cardinality, we define the uniform probability measure $\nu_H$ on $H$ by the formula
$$ \nu_H(E) \coloneqq \frac{\#(E \cap H)}{\# H}$$
(note that these two definitions are compatible with each other when $V$ is discrete).
\end{definition}

\begin{examples}  If $r>0$ and $d \geq 1$, then the ball $B(0,r) = B^d(0,r)$ is a $O(1)^d$-approximate group,  If $v \in V$ and $M$ is a natural number, then the arithmetic progression $\{mv: m=-M,\dots,M\}$ is an $O(1)$-approximate group.  If $H,H' \subset V$ are a $K$-approximate group and $K'$-approximate group respectively, then $H+H'$ is a $KK'$-approximate group.
\end{examples}

\begin{theorem}[Inverse theorem]\label{inverse}  Let $V$ be an LCA group.  Let $A,N > 0$ and $0 < \eps \leq 1/2$, and let $f \in L^{4/3}(X)_+$ for some compact subset $X$ of $V$ obeying the bounds
\begin{align}
\|f\|_{L^\infty(X)} &\leq A \label{f-infty}\\
\|f\|_{L^1(X)} &\leq AN \label{f-1}\\
\|f\|_{U^2(X)} &\geq \eps A N^{3/4}\label{f-u2}
\end{align}
(so in particular \eqref{triv-u2} is sharp up to a factor of $\eps$).
\begin{itemize}
\item[(i)]  (Truncating the small values of $f$) One has
$$ \|f 1_{f \geq \eps^4 A/16} \|_{U^2(X)} \geq \frac{\eps}{2} A N^{3/4}.$$
\item[(ii)]  (Approx. symmetry along an approx. subgroup) There exists a $\eps^{-O(1)}$-approximate group $H$ in $10X-10X$ of measure $m_V(H) = \eps^{O(1)} N$ such that
\begin{equation}\label{fnuh}
\|f * \nu_H\|_{U^2(11X-10X)} = \eps^{O(1)} A N^{3/4}.
\end{equation}
\end{itemize}
\end{theorem}

Informally, the conclusion of Theorem \ref{inverse} (when compared with \eqref{f-u2} and \eqref{fmuv}) asserts that $f*\nu_H$ resembles $f$ in some weak statistical sense, so that $f$ is ``approximately symmetric along $H$'', again in a weak statistical sense.

\begin{proof}  From \eqref{triv-u2}, \eqref{f-1} one has
$$ \|f 1_{f < \eps^4 A/16} \|_{U^2(X)} \leq \|f\|_{L^1(X)}^{3/4} (\eps^4 A/16)^{1/4} \leq \frac{\eps}{2} A N^{3/4}$$
and the claim (i) then follows from \eqref{f-u2} and the triangle inequality for the $U^2$ norm.

For (ii) our main tool will be the Balog--Szemer\'edi--Gowers theorem, to conclude that $f$ looks roughly like the function $A 1_{H+y}$ for some translate $H+y$ of an approximate group $H$, which we can then use to establish \eqref{fnuh}.

Let $F \coloneqq \{ x \in X: f(x) \geq \eps^4 A/16\}$, then $F \subset X$ and we have the pointwise bound
$$ f 1_{f \geq \eps^4 A/16} \leq A 1_F$$
and hence by \eqref{monotone}
$$ \|1_F \|_{U^2(X)} \geq \frac{\eps}{2} N^{3/4}.$$
Also, from \eqref{f-1} and Markov's inequality one has
$$ m_V(F) \leq \frac{16}{\eps^4} N.$$
From \eqref{triv-u2} one has $\|1_F \|_{U^2(X)}  \leq m_V(F)^{3/4}$.  Comparing these inequalities we conclude that
$$ m_V(F) = \eps^{O(1)} N$$
and
$$ \|1_F\|_{U^2(X)}^4 = \eps^{O(1)} N^3.$$
In the language of \cite[Definition 4.1]{tao-product} (adapted to the additive group $V$), the quantity $ \|1_F\|_{U^2(X)}^4$ is the additive energy $\Energy(F,F)$ of $F$.  Applying the Balog--Szemer\'edi--Gowers theorem (in the form of \cite[Theorem 5.4]{tao-product}), we can then find a $\eps^{-O(1)}$-approximate group $H$ in $V$ of measure
$$ m_V(H) = \eps^{O(1)} m_V(F) = \eps^{O(1)} N$$
such that
\begin{equation}\label{fey2}
m_V(F \cap (H+y)) = \eps^{O(1)} m_V(F) = \eps^{O(1)} N
\end{equation}
for some $y \in V$.  An inspection of the proof \cite[Theorem 5.4]{tao-product} also reveals that the approximate group $H$ is constructed to lie in $10F-10F$ (say), and thus also lies in $10X-10X$; in particular, $f * \nu_H$ is supported in $11X-10X$.  From \eqref{monotone}, \eqref{fey2} and the $\eps^{-O(1)}$-approximate group nature of $H$ we then have
\begin{align*}
\| f * \nu_H \|_{U^2(11X-10X)} &\geq \left\| \frac{\eps^4 A}{16} 1_{F \cap (H+y)} * \nu_H \right\|_{U^2(V)} \\
&= \eps^{O(1)} A \| 1_{F \cap (H+y)} * \nu_H * 1_{F \cap (H+y)} * \nu_H \|_{L^2(V)}^{1/2} \\
&= \eps^{O(1)} A \| 1_{F \cap (H+y)} * \nu_H * 1_{F \cap (H+y)} * \nu_H \|_{L^2(4H+2y)}^{1/2} \\
&\geq \eps^{O(1)} A m_V(4H+2y)^{-1/4} \| 1_{F \cap (H+y)} * \nu_H * 1_{F \cap (H+y)} * \nu_H \|_{L^1(4H+2y)}^{1/2} \\
&\gtrsim \eps^{O(1)} A m_V(H)^{3/4} \\
&= \eps^{O(1)} A N^{3/4}
\end{align*}
while from \eqref{fmuv}, \eqref{triv-u2} we have
$$ \|f*\nu_H\|_{U^2(11X-10X)} \leq \|f\|_{U^2(V)} \leq \|f\|_{L^1(V)}^{3/4} \|f\|_{L^\infty(V)}^{1/4} \leq A N^{3/4}$$
and the claim (ii) follows.
\end{proof}

Approximate groups $H$ (approximately) contain long arithmetic progressions $\{ mv: m = -M,\dots,M\}$ for many choices of generator $v$.  A precise formulation of this result is given by the following lemma (cf. \cite[Corollary 6.11]{dyatlov-zahl}):

\begin{lemma}[Approximate groups contain arithmetic progressions] \label{gap} Let $V$ be an LCA group isomorphic to either a lattice $\Z^d$ or a Euclidean space $\R^d$. Let $H$ be a bounded open $K$-approximate group for some $K \geq 1$.  Then for any natural number $M$, the set
$$ S \coloneqq \{ v \in V: mv \in 8H \forall m = -M,\dots,M\}$$
has measure
$$ m_V(S) \gtrsim \exp\left( - O\left( \log^{O(1)}(2K) \log(2M) \right) \right) m_V(H).$$
\end{lemma}

One can also establish this result (with a worse dependence on $M$, and with $8H$ improved to $4H$) using the Sanders-Croot-Sisask lemma \cite{sanders}, \cite{croot-sisask}, \cite{sanders-br}.  It is likely that this lemma applies to arbitrary LCA groups, and not just the lattices and Euclidean spaces, but these are the only cases we will need here.

\begin{proof}  First suppose that $V$ is isomorphic to $\Z^d$.  Applying \cite[Theorem 1.1]{sanders-br}, we see that $4H$ contains a generalized arithmetic progression $P$ of dimension $O(\log^{O(1)}(2K))$ and cardinality $\gtrsim \exp( -O(\log^{O(1)}(2K)) \# H$.  By reducing all the lengths of the progression by $M$, we see that the set
$$ \{ v \in V: mv \in P \forall m = -M,\dots,M \}$$
has cardinality
$$ \gtrsim O(M)^{-O(\log^{O(1)}(2K))} \# P \gtrsim  \exp\left( - O\left( \log^{O(1)}(2K) \log(2M) \right) \right) \# H.$$
As this set is contained in $S$, we obtain the claim (with $8H$ replaced by $4H$).

For the remaining case we may assume that $V = \R^d$ (equipped with Lebesgue measure).  Then for $\eps>0$, $2H \cap \eps \Z^d$ is a $O(K^{O(1)})$-approximate group (this follows for instance from \cite[Exercise 2.4.7]{tao-vu}), and for $\eps$ small enough the cardinality of this set is $\sim \eps^{-d} K^{O(1)} m(H)$.  Applying the preceding case, we see that
$$ \# S \cap \eps \Z^d \gtrsim  \exp\left( - O\left( \log^{O(1)}(2K) \log(2M) \right) \right) \eps^{-d} K^{O(1)} m(H).$$
Multiplying by $\eps^d$ and sending $\eps \to 0$, we obtain the claim.
\end{proof}

Given a Radon measure $\mu$ on an LCA group $V$ and a measurable set $H \subset V$, we can define the energy
$$ \Energy(\mu,H) \coloneqq 
 \mu^4\left( \{ (x_1,x_2,x_3,x_4) \in (\R^d)^4: x_1+x_2-x_3-x_4 \in H \}\right).$$
This generalizes the definition of energy in the introduction; indeed, when $V = \R^d$, we have
$$ \Energy(\mu,r) = \Energy(\mu,B(0,r)).$$
It is also invariant under affine isomorphisms $T: V \to V'$, in the sense that
\begin{equation}\label{affine}
\Energy(T_* \mu, T H) = \Energy(\mu,H)
\end{equation}
for any Radon measure $\mu$ on $V$ and measurable $H \subset V$, where $T_* \mu$ is the pushforward of $\mu$ by $T$.

We now relate these energies to Gowers norms:

\begin{lemma}\label{relate}  Let $V$ be an LCA group, let $\mu$ be a Radon measure on $V$, let $f \in L^{4/3}(V)_+$, and let $H$ be a $K$-approximate group in $V$ of positive finite measure for some $K \geq 1$.
\begin{itemize}
\item[(i)]  (Relation between energy and Gowers norm) One has
$$ \Energy(\mu,4H) = K^{O(1)} m_V(H) \| \mu * \nu_H \|_{U^2(V)}^4.$$
and for any integer $m \geq 4$, one has
$$ \Energy(\mu,mH) = K^{O(m)} \Energy(\mu,4H).$$
In particular, for any $m \geq 1$ one has
$$ \| \mu * \mu_{mH} \|_{U^2(V)} = K^{O(m)} \| \mu * \mu_{H} \|_{U^2(V)}.$$
(Compare with Lemma \ref{agn}.)
\item[(ii)]  (Shrinking the approximate symmetry group) If $m \geq 1$ is an integer and $P$ is a $K$-approximate group in $mH$ that is either finite or has positive finite measure, one has
$$ \| f * \nu_H \|_{U^2(V)} \leq K^{O(m)} \| f * \nu_P \|_{U^2(V)}.$$
\item[(iii)]  (Splitting at scale $H$)  One has
$$ \|f\|_{U^2(V)} \leq m_V(H)^{-3/4} \| f_{4H} \|_{U^2(V)}$$
where $f_{4H}$ is the local Gowers uniformity norm at scale $H$, defined by
$$ f_{4H}(y) \coloneqq \| f \|_{U^2(4H+y)}.$$
\end{itemize}
\end{lemma}

Lemma \ref{relate}(iii) asserts, roughly speaking, that to control the ``global'' $U^2$ norm of a function $f$ on an LCA group $V$, it first suffices to control the ``local'' $U^2$ norm on the ``cosets'' $4H+y$ of the approximate group $H$, and then take a further $U^2$ norm over the translation parameter $y$ (and apply a suitable normalization).  This fact can be used to give a slightly different proof of Proposition \ref{en-nearby} (setting $V=\R$ and $H=[-r,r]$), but we will not do so here.  We will eventually apply Lemma \ref{relate}(iii) to a somewhat unusual approximate group $H$, namely the sum of a ball and an arithmetic progression.

\begin{proof}  
From the Fubini--Tonelli theorem, we may expand
$$
\| \mu * \nu_{H} \|_{U^2}^4 = m_V(H)^{-4} \int_{V^4} 1_H*1_H*1_H*1_H(x_1+x_2-x_3-x_4)\ d\mu(x_1) \dots d\mu(x_4).$$
From the pointwise upper bound
$$ 1_H*1_H*1_H*1_H \leq m_V(H)^3 1_{4H}$$
we conclude that
$$
\| \mu * \nu_{H} \|_{U^2}^4 \leq m_V(H)^{-1} \Energy(\mu, 4H).$$
Conversely, from the pointwise lower bound
\begin{equation}\label{4h}
1_{4H} * 1_{4H} * 1_{4H} * 1_{4H} \geq m_V(H)^3 1_H
\end{equation}
we see that
$$
\| \mu * \nu_{4H} \|_{U^2}^4 \geq m_V(H)^{-1} \Energy(\mu, H).$$
Replacing $H$ by $mH$ for any natural number $m$ we conclude that
$$ 
\| \mu * \nu_{mH} \|_{U^2}^4 \leq m_V(H)^{-1} \Energy(\mu, 4mH)$$
and
$$  \Energy(\mu, mH) \leq K^{O(m)} m_V(H) \| \mu * \nu_{4mH} \|_{U^2}^4.$$
Next, we can upper bound $\nu_{4mH}$ by the sum of at most $K^{O(m)}$ translates of $\nu_H$, thus by the triangle inequality
$$ \| \mu * \nu_{4mH} \|_{U^2}^4 \leq K^{O(m)} \| \mu * \mu_{H} \|_{U^2}^4.$$
Combining these inequalities we obtain (i).

For (ii), we observe the pointwise bound
$$ \nu_H \leq \frac{m_V((m+1) H)}{m_V(H)} \nu_{(m+1)H} * \nu_P$$
and thus by \eqref{fmuv}
$$ \| f * \nu_H \|_{U^2(V)} \leq \frac{m_V((m+1) H)}{m_V(H)} \| f * \nu_P \|_{U^2(V)}.$$
The claim (ii) now follows from the $K$-approximate group nature of $H$.

Now we prove (iii).  From \eqref{4h} we have
$$ \| 1_{4H} \|_{U^2(V)} \geq m_V(H)^{3/4}$$
and hence by \eqref{fub-ton} we have
$$ \| f \otimes 1_{4H} \|_{U^2(V \times V)} \geq m_V(H)^{3/4} \|f\|_{U^2(V)}.$$
Applying the measure-preserving transformation $(x,h) \mapsto (x+h,h)$ on $V \times V$ we conclude that
$$ \| F \|_{U^2(V \times V)} \geq m_V(H)^{3/4} \|f\|_{U^2(V)}$$
where $F(x,h) \coloneqq f(x+h) 1_{4H}(h)$.  Applying Lemma \ref{split}, we have
$$\| F \|_{U^2(V \times V)}  \leq \| f_{4H} \|_{U^2(V)},$$
and the claim (iii) follows.
\end{proof}

\section{The higher dimensional case}\label{higher-sec}

We are now ready to establish Theorem \ref{main-second}.  By repeating the arguments in Section \ref{induct-sec} (which extend to higher dimensions without difficulty), it suffices to establish the following proposition.

\begin{proposition}[Slight gain over the trivial bound]\label{slight-gain}   
Let $d \geq 1$ be an integer, let $0 < \delta < d$ be a non-integer, and let $C>1$ and $0 < \eps \leq 1/2$.  Let $0 < r_0 < 1$ be sufficiently small depending on $d,\delta,C,\eps$.  Let $X \subset \R^d$ be a $\delta$-regular set on scales $[r_0,1]$ with constant $C$, and let $\mu_X$ be an associated regular measure.  Then we have
$$ \Energy(\mu|_{B(0,1)}, r_0) \leq \eps r_0^\delta.$$
In fact one can take $r_0$ to be quasipolynomial in $C/\eps$, in the sense that
\begin{equation}\label{r0-size} 
r_0 = \exp\left( - \exp\left( O_{\delta,d}\left( \log^{O_{\delta,d}}(C/\eps) \right) \right) \right).
\end{equation}
\end{proposition}

To prove this proposition we will use induction on the ambient dimension $d$.  To facilitate the induction, it is convenient to relax the hypotheses on $\mu$ somewhat.  More precisely, our induction hypothesis will be as follows.

\begin{proposition}[Induction hypothesis]\label{induct-hyp}  
Let $d \geq 1$ be an integer, let $0 < \delta < d$ be a non-integer, and let $C>1$ and $0 < \eps \leq 1/2$.  Let $r_0$ be the quantity \eqref{r0-size}.  Let $X \subset B(0,1)$ be a compact set with the following property:
\begin{itemize}
\item[(i)]  (Upper $\delta$-regularity) For any $r_0 \leq r_2 \leq r_1 \leq 1$ and any $x \in \R^d$, the set $X \cap B(x,r_1)$ can be covered by at most $C (r_1/r_2)^\delta$ balls of radius $r_2$.
\end{itemize}
Let $\mu$ be a Radon measure supported on $X$ obeying the upper regularity bound
\begin{equation}\label{crd}
 \mu(B(x,r_1)) \leq C r_1^\delta
\end{equation}
for all $r_0 \leq r_1 \leq 1$.  Then
$$ \Energy(\mu, r_0) \leq \eps r_0^\delta.$$
\end{proposition}

The main advantage of working with Proposition \ref{induct-hyp} instead of with Proposition \ref{slight-gain} is that all the hypotheses on $\mu$ and $X$ are of ``upper bound'' type rather than ``lower bound'' type, so that it becomes easier to remove unwanted portions of $\mu$ or $X$ as needed.

A simple covering argument shows that if $X$ is a $\delta$-regular set on scales $[r_0,1]$ then the property (i) of Proposition \ref{induct-hyp} is satisfied with $C$ replaced by $O_{\delta,d,C}(1)$, and the property \eqref{crd} is immediate from the definition of a $\delta$-regular measure.  Hence the general dimension case of Proposition \ref{slight-gain} follows from Proposition \ref{induct-hyp}.

To prove Proposition \ref{induct-hyp}, we induct on $d$, assuming that the claim has already been proven for dimension $d-1$ (this induction hypothesis is vacuous when $d=1$).  Let $0 < \delta < d$ be non-integer, let $C>1$ and $\eps>0$, and let $r_0>0$ be chosen later (eventually it will be of the form \eqref{r0-size}).  Let $X, \mu_X$ obey the hypotheses of Proposition \ref{induct-hyp}.  We assume for sake of contradiction that
\begin{equation}\label{start}
\Energy(\mu, r_0) > \eps r_0^\delta.
\end{equation}
To abbreviate notation we now allow all implied constants in asymptotic notation to depend on $\delta,d$.
From Lemma \ref{relate}(i) we conclude that
$$ \| \mu * \nu_{B(0,r_0)} \|_{U^2(B(0,2))} \gtrsim \eps^{O(1)} r_0^{(\delta-d)/4}.$$
Informally, this estimate asserts that $X + B(0,r_0)$ has high additive energy.
We now use the inverse theory to obtain some regularity along a non-trivial arithmetic progression $P$.
From \eqref{crd} we also have the uniform bound
$$\| \mu * \nu_{B(0,r_0)} \|_{L^\infty(B(0,2))} \lesssim C r_0^{\delta-d}$$
and from \eqref{crd} (with $r=1$) and Young's inequality one has
$$\| \mu * \nu_{B(0,r_0)} \|_{L^1(B(0,2))} \lesssim C.$$
We can now apply Theorem \ref{inverse}(ii) to conclude that there exists a $(C/\eps)^{O(1)}$-approximate group $H$ in $B(0,20)$ of measure $|H| = (C/\eps)^{O(1)} r_0^{d-\delta}$ such that
\begin{equation}\label{miro}
\| \mu * \nu_{B(0,r_0)} * \nu_H \|_{U^2(B(0,20))} = (C/\eps)^{O(1)} r_0^{(\delta-d)/4}.
\end{equation}
Informally, this estimate asserts that $X + B(0,r_0)$ not only has high additive energy, but also behaves like a union of translates of $B(0,r_0) + H$.

Let $M$ be a large integer (depending on $C,d,\delta,\eps$) to be chosen later.  By Lemma \ref{gap}, the set
$$ S \coloneqq \{ v \in \R^d: mv \in 8H \forall m = -10^3 M,\dots,10^3 M\}$$
has measure
$$ |S| \gtrsim M^{-O(\log^{O(1)}(C/\eps))}  |H| \sim M^{-O(\log^{O(1)}(C/\eps))} r_0^{d-\delta}.$$
In particular, since $\delta>0$, we have $|S| > |B(0,r_0)|$ if $r_0$ is sufficiently small, and more specifically we can take $r_0$ of the form
\begin{equation}\label{r0-M}
r_0 = M^{-O(\log^{O(1)}(C/\eps))}.
\end{equation}
We conclude that there is a vector $v \in \R^d$ with $|v| > r_0$ such that the progression
$$ P \coloneqq \{ mv: m = -M,\dots, M \} $$
is such that $10^3 P \subset 8H$.  In particular we must have the scale relations
$$ r_0 \leq |v| \leq Mv \leq 1.$$
From \eqref{miro} and Lemma \ref{relate}(ii) we conclude that
\begin{equation}\label{muppet}
\| \mu * \nu_{B(0,r_0)} * \nu_P \|_{U^2(B(0,20))} \gtrsim (C/\eps)^{O(1)} r_0^{(\delta-d)/4}.
\end{equation}
Informally, this estimate asserts that $X + B(0,r_0)$ not only has high additive energy, but also behaves like a union of translates of the set $B(0,r_0) + P$, which is an arithmetic progression of small balls.

Now that we have obtained some regularity along a progression $P$, 
the next step is to localize at the scale $M|v|$ of the diameter of $P$.  For any $y \in \R^d$, the quantity
\begin{equation}\label{quant}
 \| \mu * \nu_{B(0,r_0)} * \nu_P \|_{U^2(B(y,M|v|))}
 \end{equation}
vanishes unless $y$ lies within $O(M|v|)$ of $X$, which constraints $y$ to a set of measure at most $O( (M|v|)^{d-\delta})$ thanks to the hypothesis (i).  From \eqref{triv-u2} we conclude that the $U^2$ norm of the quantity \eqref{quant} (viewed as a function of $y$) is at most
$$ \lesssim (M|v|)^{3(d-\delta)/4} \sup_{y \in \R^d}  \| \mu * \nu_{B(0,r_0)} * \nu_P \|_{U^2(B(y,M|v|))}$$
and hence by Lemma \ref{relate}(iii) we have
$$ 
\|\mu * \nu_{B(0,r_0)} * \nu_P\|_{U^2(B(0,20))}
\lesssim (C/\eps)^{O(1)} (M|v|)^{-3\delta/4}  \sup_{y \in \R^d}  \| \mu * \nu_{B(0,r_0)} * \nu_P \|_{U^2(B(y,M|v|))}.$$
Comparing this with \eqref{muppet}, we conclude that there exists $y_0 \in \R^d$ such that
$$
\| \mu * \nu_{B(0,r_0)} * \nu_P \|_{U^2(B(y_0,M|v|))}
\gtrsim (C/\eps)^{O(1)} (M|v|)^{3\delta/4} r_0^{(\delta-d)/4}.$$
Informally, this estimate asserts that $(X + B(0,r_0)) \cap B(y_0, M|v|)$ not only has high additive energy, but also behaves like a union of translates of the set $B(0,r_0) + P$, which has diameter comparable to that of the ball $B(y_0, M|v|)$.

Fix this $y_0$.  We now apply Lemma \ref{relate}(iii) using the $O(1)$-approximate group
$$ H \coloneqq B(0,|v|) + P$$
(which geometrically is approximately a cylinder of dimensions $|v| \times M|v|$, oriented in the direction of $v$) to conclude that
$$
\| \mu * \nu_{B(0,r_0)} * \nu_P \|_{U^2(B(y_0,M|v|))}
\lesssim (C/\eps)^{O(1)} (M |v|^d)^{-3/4} \|f\|_{U^2(\R^d)}$$
where $f(y)$ is the local Gowers norm
$$ f(y) \coloneqq \| \mu * \nu_{B(0,r_0)} * \nu_P \|_{U^2(B(y_0,M|v|) \cap (B(y,4|v|)+4P )}.$$
The function $f$ vanishes unless $y \in B(y_0,10M|v|)$, thus
\begin{equation}\label{fu2}
 \|f\|_{U^2(B(y_0,10M|v|))}
\gtrsim (C/\eps)^{O(1)} M^{3(\delta+1)/4} |v|^{3(\delta+d)/4} r_0^{(\delta-d)/4}.
\end{equation}
Informally, this estimate asserts that the collection of cosets $H+y$ of the ``cylinder'' $H$ in which $(X + B(0,r_0)) \cap B(y_0, M|v|)$ has high additive energy, itself has high additive energy.

We now study the quantity $f(y)$ for some $y \in (y_0,10M|v|)$.  We can use monotonicity and the triangle inequality to bound
\begin{align*}
 f(y) &\leq \| \mu|_{B(y,5|v|)+5P} * \nu_{B(0,r_0)} * \nu_P \|_{U^2(\R^d)} \\
 &\leq \sum_{z \in y+5P} \| \mu|_{B(z,5|v|)} * \nu_{B(0,r_0)} * \nu_P \|_{U^2(\R^d)}.
 \end{align*}
 The summand vanishes unless $z$ lies in $X + B(0,5|v|)$.  We also have from \eqref{crd} and Young's inequality that
$$ \| \mu|_{B(z,5|v|)} * \nu_{B(0,r_0)} * \nu_P \|_{L^1(\R^d)}
\leq \mu( B(z,5|v|) ) \lesssim (C/\eps)^{O(1)} |v|^\delta$$
and
$$ \| \mu|_{B(z,5|v|)} * \nu_{B(0,r_0)} * \nu_P \|_{L^\infty(\R^d)}
\lesssim (C/\eps)^{O(1)} M^{-1} r_0^{\delta-d}$$
and hence by \eqref{triv-u2} we have
$$ \| \mu|_{B(z,5|v|)} * \nu_{B(0,r_0)} * \nu_P \|_{U^2(\R^d)}
\lesssim (C/\eps)^{O(1)} M^{-1/4} r_0^{(\delta-d)/4} |v|^{3\delta/4}.$$
We thus have
$$ f(y) \lesssim (C/\eps)^{O(1)} M^{-1/4} r_0^{(\delta-d)/4} |v|^{3\delta/4}
\# ( (y+5P) \cap (X + B(0,5|v|) ) ).$$
If we introduce the function
$$ F(y) \coloneqq 1_{B(y_0,10M|v|) \cap (X + B(0, 10|v|))}$$
then a simple volume-packing argument shows that
$$ \# ( (y+5P) \cap (X + B(0,5|v|) ) ) \lesssim (C/\eps)^{O(1)} M F * \nu_{B(0,10|v|)+10P}(y) $$
and thus we have the pointwise bound
$$ f \lesssim (C/\eps)^{O(1)} M^{3/4} r_0^{(\delta-d)/4} |v|^{3\delta/4}
F * \nu_{B(0,10|v|)+10P}.$$
Inserting this into \eqref{fu2}, we conclude that
$$ \|F * \nu_{B(0,10|v|)+10P} \|_{U^2(\R^d)}
\gtrsim (C/\eps)^{O(1)} M^{3\delta/4} |v|^{3d/4}.$$
Clearly we have
$$\|F * \nu_{B(0,10|v|)+10P} \|_{L^\infty(\R^d)} \leq \|F\|_{L^\infty(\R^d)} \leq 1$$
while from the hypothesis (i) the support of $F$ is covered by $O(M^\delta)$ balls of radius $|v|$, and hence
\begin{equation}\label{fb}
\|F * \nu_{B(0,10|v|)+10P} \|_{L^1(\R^d)} \leq \|F\|_{L^1(\R^d)} \lesssim M^\delta |v|^d.
\end{equation}
Applying Theorem \ref{inverse}(i), we conclude that
\begin{equation}\label{fog}
 \| (F * \nu_{B(0,10|v|)+10P}) 1_G \|_{U^2(\R^d)} 
\gtrsim (C/\eps)^{O(1)} M^{3\delta/4} |v|^{3d/4}
\end{equation}
where $G$ is a set of the form
$$ G \coloneqq \{ y: F * \nu_{B(0,10|v|)+10P}(y) \geq (C/\eps)^{-C_0}\}$$
for some sufficiently large $C_0$ depending only on $d,\delta$.  Note that $G$ is a compact subset of $B(y_0,100M|v|)$.  From \eqref{fb}, \eqref{fog} we conclude that
\begin{equation}\label{1g}
\|1_G \|_{U^2(B(y_0,100M|v|))} \gtrsim (C/\eps)^{O(1)} M^{3\delta/4} |v|^{3d/4}.
\end{equation}
Informally, this estimate asserts that the collection of cosets $H+y$ of the ``cylinder'' $H$ in which $(X + B(0,r_0)) \cap B(y_0, M|v|)$ has large density, itself has high additive energy.

We now split into two cases: the low-dimensional case $\delta < 1$ and the high-dimensional case $\delta > 1$ (recall that $\delta$ is assumed to be non-integer).  In the low-dimensional case we observe that for any $y \in \R^d$, the only portion of $X$ that contributes to $F * \nu_{B(0,10|v|)+10P}(y)$ lies in $B(y, 100 M |v|)$, and is thus covered by $O(M^\delta)$ balls of radius $|v|$ thanks to the hypothesis (i).  This leads to the pointwise estimate
$$ F * \nu_{B(0,10|v|)+10P}(y) \lesssim (C/\eps)^{O(1)} M^{-1} M^\delta;$$
as we are in the low-dimensional case $\delta<1$, taking $M = (C/\eps)^{C_1}$ for a sufficiently large $C_1$ (depending on $d,\delta$) sufficiently large will then imply that the set $G$ is empty, which contradicts \eqref{1g}, with $r_0$ of the required size thanks to \eqref{r0-M}.  (Informally, the point is that the set $X$ is too low dimensional to adequately fill out a coset $H+y$.)

Now suppose we are in the high-dimensional case $\delta>1$, which of course forces $d \geq 2$.  Here we shall ``quotient out'' by $P$ and use the induction hypothesis.

Let $v^\perp \coloneqq \{ x \in \R^d: x \cdot v = 0\}$ be the hyperplane in $\R^d$ orthogonal to $v$, and let $\pi \colon \R^d \to v^\perp$ be the orthogonal projection.  

As $G$ is a compact subset of $B(y_0,100M|v|)$, $\pi(G)$ is a compact subset of the disk $B_{v^\perp}(\pi(y_0), 100M|v|)$, where we use $B_{v^\perp}$ to denote the balls in the hyperplane $v^\perp$.  After applying a rigid motion to identify $v^\perp$ with $\R^{d-1}$, one can view $G$ as a subset of $\pi(G) \times [-100M|v|, 100M|v|]$, hence by \eqref{monotone}, \eqref{fub-ton}, \eqref{triv-u2} we have
$$ \|1_G \|_{U^2(B(y_0,100M|v|))} \lesssim  (C/\eps)^{O(1)} (M|v|)^{3/4} \|1_{\pi(G)} \|_{U^2(B_{v^\perp}(\pi(y_0), 100M|v|))}$$
and thus by \eqref{1g}
\begin{equation}\label{1-pig}
 \|1_{\pi(G)} \|_{U^2(B_{v^\perp}(\pi(y_0), 100M|v|))} \gtrsim  (C/\eps)^{O(1)} M^{3(\delta-1)/4} |v|^{3(d-1)/4}.
 \end{equation}
 Informally, this estimate asserts that the set $G$ (which roughly speaking tracked the translates of $H+y$ in which $(X + B(0,r_0)) \cap B(y_0, M|v|)$ had large density) continues to have large additive energy after applying the orthogonal projection $\pi$.

The set $\pi(G)$ also has ``dimension $\delta-1$'' in the scale range $[|v|, M|v|]$ in the following sense:

\begin{lemma}[$\pi(G)$ is $\delta-1$-dimensional]\label{delta}  For any $|v| \leq r_2 \leq r_1 \leq M|v|$ and any $x \in v^\perp$, the set $\pi(G) \cap B_{v^\perp}(x,r_1)$ can be covered by at most $ (C/\eps)^{O(1)} (r_1/r_2)^{\delta-1}$ balls of radius $r_2$ in $v^\perp$.
\end{lemma}

\begin{proof}  Let $\Sigma$ be any $100 r_2$-separated subset of $\pi(G) \cap B_{v^\perp}(x,r_1)$.  It will suffice to show that $\# \Sigma \lesssim  (C/\eps)^{O(1)} (r_1/r_2)^{\delta-1}$.  Each $\sigma \in \Sigma$ can be written as $\pi(y_\sigma)$ for some $y_\sigma \in G$.

By construction of $G$ and $F$, we have
$$ |(B(y_\sigma,10|v|) + 10P) \cap (X + B(0, 20|v|))| \gtrsim  (C/\eps)^{O(1)} M |v|^d$$
for all $\sigma \in \Sigma$.  In particular, it requires at least $\gtrsim  (C/\eps)^{O(1)} M|v| / r_2$ balls of radius $r_2$ to cover this set.  As these sets are at least $10r_2$-separated (say) from each other as $\sigma$ varies, we conclude that the set
$$ \bigcup_{\sigma \in \Sigma} (B(y_\sigma,10|v|) + 10P) \cap (X + B(0, 20|v|))$$
requires at least $\gtrsim  (C/\eps)^{O(1)} \frac{M|v|}{r_2} \# \Sigma$ balls of radius $r_2$ to cover this set.  On the other hand, this set is contained in
\begin{equation}\label{xbb}
 (X + B(0,20|v|)) \cap B( y_0, 200M|v|) \cap \pi^{-1}( B_{v^\perp}(x,20r_1)).
 \end{equation}
 The set $B( y_0, 200M|v|) \cap \pi^{-1}( B_{v^\perp}(x,20r_1))$ (which is shaped roughly like a $r_1 \times M|v|$ cylinder) can be covered by $O( M|v| / r_1 )$ balls of radius $r_1$.  Applying hypothesis (i), we conclude that the set \eqref{xbb} can be covered by $O(  (C/\eps)^{O(1)} (M|v|/r_1) (r_1/r_2)^\delta )$ balls of radius $r_2$.  Comparing these bounds, we obtain the claim.
 \end{proof}
 
 Let $\mu'$ denote the measure
$$ \mu' \coloneqq |v|^{1-d} M^{1-\delta} m_{v^\perp}|_{\pi(G)+ B_{v^\perp}(0,|v|)}$$
where $m_{v^\perp}$ is Lebesgue measure on $v^\perp$.  From Lemma \ref{delta} (taking $r_2 = |v|$) we see that
$$ \mu'( B_{v^\perp}(y, r) ) \lesssim  (C/\eps)^{O(1)} (r / M|v|)^{\delta-1}$$
for all $r \in [|v|, M|v|]$.  Applying the induction hypothesis Proposition \ref{induct-hyp} (with $d$ replaced by $d-1$, $\delta$ replaced by $\delta-1$, and $r_0$ replaced by $1/M$) and an affine change of variable to rescale $v^\perp$ to $\R^{d-1}$ and $B_{v^\perp}(y_0,100M|v|)$ to $B^{d-1}(0, 1)$, we conclude from \eqref{affine} that
$$ \Energy(\mu', |v|) \leq \eps' M^{-(\delta-1)}$$
for any $\eps'>0$, if $M$ is sufficiently large depending on $C,\delta,d,\eps, \eps'$; indeed we can take $M$ to be of the form
$$ M = \exp\left( \exp( O( \log^{O(1)}(C/\eps\eps') ) ) \right).$$
Applying Lemma \ref{relate}(i), we conclude that
$$ \| \mu' * \nu_{B_{v^\perp}(0, |v|)} \|_{U^2(v^\perp)} \lesssim (C/\eps)^{O(1)} (\eps')^{1/4} |v|^{-(d-1)/4} M^{-(\delta-1)/4}.$$
Using the pointwise bound
$$
\mu' * \nu_{B_{v^\perp}(0, |v|)} \gtrsim (C/\eps)^{O(1)} |v|^{1-d} M^{1-\delta} 1_{\pi(G)}$$
we conclude that
$$ \| 1_{\pi(G)} \|_{U^2(v^\perp)} \lesssim (C/\eps)^{O(1)} (\eps')^{1/4} |v|^{3(d-1)/4} M^{3(\delta-1)/4}.$$
By taking $\eps'$ to equal $(\eps/C)^{C_2}$ for a sufficiently large $C_2$ depending only on $d,\delta$, we contradict \eqref{1-pig}, giving the claim, with the right size \eqref{r0-size} for $r_0$ thanks to \eqref{r0-M} and the choice of parameters $\eps',M$.

\section{Nonlinear expansion}\label{nonlinear-sec}

In this section we establish Theorem \ref{nonlinear}.  We establish the higher dimensional case $d>1$ here; the one-dimensional case $d=1$ is proven similarly.  The strategy is to localize to a small enough scale that the nonlinear function $F$ can be well approximated by a linear one, at which point one can apply the Cauchy--Schwarz inequality to control the relevant expressions by additive energies, so that Theorems \ref{main}, \ref{main-second} may be applied.

We allow all constants to depend on $\delta,d,F$.  Let $\eps>0$ be a small quantity depending on $\delta,d,F$ to be chosen later; by shrinking $r$ as necessary we may assume that $r$ is sufficiently small depending on $\eps,\delta,d,F$.  Let $y_0$ be a point in $Y$, thus $y_0 \in B(0,2)$.  We will only work in the $\eps r^{1/2}$-neighborhood of $y_0$ in $Y_r$.  Indeed, it will suffice to establish the bound
$$ |F( X_r, Y_r \cap B(y_0, \eps r^{1/2}) )| \gtrsim_{C,\eps} r^{d-\delta-\beta}.$$
By regularity, we can cover $X_r$ by $O_\eps(r^{-\delta/2})$ balls $B(x_i, \eps r^{1/2})$ of radius $\eps r^{1/2}$ with $x_i \in X \cap B(0,2)$.  Using Taylor expansion, the $C^2$ nature of $F$ and the non-vanishing of the differential map $D_x F(x,y)$, we see that any two sets
$$ F( B(x_i,\eps r^{1/2}), B(y_0, \eps r^{1/2})), F( B(x_j,\eps r^{1/2}), B(y_0, \eps r^{1/2}))$$
with $|x_i-x_j| \leq \eps$ will be disjoint unless $|x_i-x_j| \lesssim \eps r^{1/2}$.  From this we conclude that the sets
$F( B(x_i,\eps r^{1/2}), B(y_0, \eps r^{1/2}))$ have overlap $O_\eps(1)$, so it will suffice to establish the bound
$$ |F( X_r \cap B(x_0, \eps r^{1/2}), Y_r \cap B(y_0, \eps r^{1/2}) )| \gtrsim_{C,\eps} r^{d-\delta/2-\beta}$$
for each $x_0 \in X \cap B(0,2)$.

By Taylor expansion and the $C^2$ nature of $F$, for $x \in B(x_0,\eps r^{1/2})$ and $y \in B(y_0,\eps r^{1/2})$ one has
$$ F(x,y) = F(x_0,y_0) + D_x F(x_0,y_0) (x-x_0) + D_y F(x_0,y_0) (y-y_0) + O(\eps r).$$
From this and the invertibility of $D_x F(x_0,y_0), D_y F(x_0,y_0)$ we see that the set
$$
F( X_r \cap B(x_0, \eps r^{1/2}), Y_r \cap B(y_0, \eps r^{1/2}) )$$
contains the set
$$ F(x_0,y_0) + 
D_x F(x_0,y_0) ((X_{r/2} \cap B(x_0, \eps r^{1/2})) - x_0)
+ D_y F(x_0,y_0) ((Y_{r/2} \cap B(y_0, \eps r^{1/2})) -y_0)$$
so after subtracting a constant and inverting the linear transformation $D_x F(x_0,y_0)$ it will suffice to establish the bound
\begin{equation}\label{ab}
|A + B| \gtrsim_{C,\eps} r^{d-\delta/2 - \beta}
\end{equation}
where
$$ A \coloneqq X_{r/2} \cap B(x_0, \eps r^{1/2})$$
and
$$ B \coloneqq D_x F(x_0,y_0)^{-1} D_y F(x_0,y_0) (Y_{r/2} \cap B(y_0, \eps r^{1/2})).$$
From the regular nature of $X$ one sees from Definition \ref{reg-set} and standard volume-packing calculations that
$$ |A| \sim_{C,\eps} r^{d-\delta/2}$$
and a similar argument (using also the invertibility of $D_y F$ and the $C^2$ nature of $F$) gives
\begin{equation}\label{b-card}
|B| \sim_{C,\eps} r^{d-\delta/2}.
\end{equation}
In particular
$$ \| 1_A * 1_B \|_{L^1(\R^d)} \sim_{C,\eps} r^{2(d-\delta/2)}.$$
Since $1_A*1_B$ is supported on $A+B$, to prove \eqref{ab} it thus suffices by Cauchy-Schwarz to show that
$$ \| 1_A * 1_B \|_{L^2(\R^d)}^2 \lesssim_{C,\eps} r^{3(d-\delta/2)+\beta}.$$
By the Gowers-Cauchy-Schwarz inequality we have
$$\| 1_A * 1_B \|_{L^2(\R^d)}^2 \leq \|1_A \|_{U^2(\R^d)}^2 \|1_B \|_{U^2(\R^d)}^2.$$
From \eqref{young}, \eqref{b-card} we have
$$ \|1_B \|_{U^2(\R^d)}^4 \lesssim_{C,\eps} r^{3(d-\delta/2)}$$
so it will suffice to establish the bound
$$ \|1_A \|_{U^2(\R^d)}^4 \lesssim_{C,\eps} r^{3(d-\delta/2) + \beta}.$$
From the pointwise estimate
$$ 1_A \lesssim 1_{X_r \cap B(x_0,r^{1/2})} * \nu_{B(0,r)}$$ 
and Lemma \ref{relate}(i) (or the higher-dimensional version of Lemma \ref{agn}) it suffices to show that
$$ \Energy(1_{X_r \cap B(x_0,r^{1/2})} dm,r) \lesssim_{C,\eps}  r^{4d-3\delta/2 + \beta}$$
(where $dm$ is Lebesgue measure).  We rescale this as
$$ \Energy(r^{(\delta-d)/2} 1_{\tilde X_{r^{1/2}} \cap B(0,1)} dm,r) \lesssim_{C,\eps}  r^{\delta/2 + \beta}$$
where $\tilde X$ is a rescaled version of $X$:
$$ \tilde X \coloneqq \left\{ \frac{x-x_0}{r^{1/2}}: x \in X \right\}.$$
From the regularity of $X$ and a routine change of variables we check that $\tilde X$ is $\delta$-regular at scales $[r^{1/2},1]$ (with constant $C^{O(1)}$), and that $r^{(\delta-d)/2} 1_{\tilde X_{r^{1/2}} \cap B(0,1)} dm$ is an associated regular measure (again with constant $C^{O(1)}$).  The claim now follows from Theorem \ref{main-second} (adjusting the constants in the definition of $\beta$ appropriately).

\begin{remark}  The regularity hypotheses on $Y$ can be relaxed substantially; in fact with a little more effort one could replace $Y_r$ here by any subset of $B(0,1)$ of measure $\gtrsim_{C,\delta} r^{d-\delta}$.  We leave the details to the interested reader.
\end{remark}

\section{From additive energy to the fractal uncertainty principle}\label{fract-sec}

We now prove Theorem \ref{add-eng}.  Let the notation and hypotheses be as in that theorem.  Clearly we have
$$ \| {\mathcal F}_h 1_{Y_h} \|_{L^1(\R^d) \to L^\infty(\R^d)} \lesssim_d h^{-d/2}$$
so by the Riesz--Thorin theorem it suffices to show that
$$ \| {\mathcal F}_h 1_{Y_h} \|_{L^\infty(\R^d) \to L^4(\R^d)} \lesssim_d h^{d/2-3\delta/4+2\beta}.$$
Let $f \in L^\infty(\R)$ be of norm one.  From \eqref{u24-v} and a rescaling we have
$$
\| {\mathcal F}_h (f 1_{Y_h}) \|_{L^4(\R^d)} \sim_d h^{d/4} \| f 1_{Y_h} \|_{U^2(\R^d)}.$$ 
As $Y$ is $\delta$-regular, we have the pointwise bound
$$ f 1_{Y_h} \lesssim_d C h^{d-\delta} \mu_Y * \nu_{B(0,2h)}$$
and hence 
$$\| {\mathcal F}_h (f 1_{Y_h}) \|_{L^4(\R^d)} \lesssim_d C h^{3d/4-\delta} \| \mu_Y * \nu_{B(0,2h)} \|_{U^2(\R^d)}.$$
Applying Lemma \ref{relate}(i) we conclude that
$$\| {\mathcal F}_h (f 1_{Y_h}) \|_{L^4(\R^d)} \lesssim_d C h^{d/2-\delta} \Energy( \mu_Y, \delta )^{1/4}$$
and the claim \eqref{fey} now follows from Theorems \ref{main}, \ref{main-second}, after adjusting $\beta$ as necessary.  The claim \eqref{fey-2} then follows from H\"older's inequality after observing from the $\delta'$-regularity of $X$ that
$$ |X_h| \lesssim_{C',\delta',d} h^{d-\delta'}.$$


\end{document}